\documentclass[12pt] {amsart}

\usepackage{amsmath}
\usepackage{amsfonts}
\usepackage{amssymb}
\usepackage{latexsym}
\newtheorem{theorem}{Theorem}[section]
\newtheorem{lemma}[theorem]{Lemma}
\newtheorem{proposition}[theorem]{Proposition}
\newtheorem{corollary}[theorem]{Corollary}

\newtheorem{remark}[theorem]{Remark}
\numberwithin{equation}{section}

\def\sD{{\mathfrak D}}      
   \def\sH{{\mathfrak H}}   
      \def\sL{{\mathfrak L}}
\def\sM{{\mathfrak M}}   \def\sN{{\mathfrak N}}

      \def\dC{{\mathbb C}}

   \def\dN{{\mathbb N}}   
      \def\dR{{\mathbb R}}

\def\cA{{\mathcal A}}   \def\cB{{\mathcal B}}   \def\cC{{\mathcal C}}
\def\cD{{\mathcal D}}      
   \def\cH{{\mathcal H}}   
   \def\cK{{\mathcal K}}   \def\cL{{\mathcal L}}
   \def\cN{{\mathcal N}}   
      \def\cR{{\mathcal R}}
      \def\cU{{\mathcal U}}

\def\cdom{{\rm \overline{dom}\,}}
\def\bB{{\mathbf B}}

\def\RE{{\rm Re\,}}

\def\wt{\widetilde}
\def\wh{\widehat}

\def\f{\varphi}
\def\uphar{{\upharpoonright\,}}

\def\ran{{\rm ran\,}}
\def\dom{{\rm dom\,}}
\def\cran{{\rm \overline{ran}\,}}

\usepackage{color}

\def\*{{\phantom *}}

\begin{document}
\thispagestyle{empty}
\phantom{.} 

\vskip 0.5cm

\title
 {Around the Van Daele--Schm\"{u}dgen theorem}

\author{Yury~Arlinski\u{i}}
\address{Department of Mathematical Analysis \\
East Ukrainian National University \\
Kvartal Molodyozhny 20-A \\
Lugansk 91034 \\
Ukraine} \email{yury.arlinskii@gmail.com}

\author{Valentin A.~Zagrebnov}
\address{D\'{e}partement de Math\'{e}matiques -
Universit\'{e} d'Aix-Marseille and Institut de Math\'{e}matiques de
Marseille (LATP) - UMR 7353, CMI - Technop\^{o}le Ch\^{a}teau-Gombert, 39
rue F. Joliot Curie, 13453 Marseille Cedex 13, France}
\email{Valentin.Zagrebnov@univ-amu.fr}



\begin{abstract}For a \textit{bounded} non-negative self-adjoint operator acting in a complex, infinite-dimensional,
separable Hilbert space $\cH$ and possessing a dense range $\cR$  we
propose a new approach to characterisation of phenomenon concerning
the existence of subspaces $\sM\subset \cH$ such that
$\sM\cap\cR=\sM^\perp\cap\cR=\{0\}$. We show how the existence of
such subspaces leads to various \textit{pathological} properties of
\textit{unbounded} self-adjoint operators related to von Neumann theorems \cite{Neumann}--\cite{Neumann2}.
{We revise the von Neumann-Van Daele-Schm\"{u}dgen assertions \cite{Neumann}, \cite{Daele}, \cite{schmud}
to refine them.}
We also develop {a new systematic approach,} which allows to construct for any \textit{unbounded}
densely defined symmetric/self-adjoint operator $T$ \textit{infinitely} many pairs $\langle T_1 , T_2 \rangle$
of its closed densely defined restrictions $T_k\subset T$ such that
$\dom(T^* T_{k})=\{0\}$ ($\Rightarrow\dom T_{k}^2=\{0\}$) $k=1,2$
and $\dom T_1\cap\dom T_2=\{0\}$, $\dom T_1\dot+\dom T_2=\dom T$.

\end{abstract}
\maketitle
\tableofcontents

\section{Introduction}\label{Int}
Throughout this paper we consider infinite-dimensional and separable Hilbert spaces over the field $\dC$ of complex numbers.
If $\cH$ is a Hilbert space, then its (proper) linear subset $\mathfrak{M} \subset\cH$ is called a \textit{linear manifold}.
The closure $\overline{\mathfrak{M}}$ in topology of $\cH$ is itself a Hilbert space. We call this \textit{closed} linear
manifold a \textit{subspace} of the space $\cH$. Let $\mathfrak{M}_1$ and
$\mathfrak{M}_2$ be linear manifolds of $\cH$. Then $\mathfrak{M}_1 + \mathfrak{M}_2$ denotes the \textit{sum} of manifolds,
which is the smallest linear manifold that contains $\mathfrak{M}_1$ and $\mathfrak{M}_2$. If intersection of subsets
$\mathfrak{M}_1$ and $\mathfrak{M}_2$ has only \textit{zero vector} in common, we denote
the sum by $\mathfrak{M}_1 \dot + \mathfrak{M}_2$ and call it the \textit{direct sum} of linear manifolds. If in addition
these two linear manifolds are mutually orthogonal, then we denote their sum
as $\mathfrak{M}_1 \oplus \mathfrak{M}_2$ and we call it the \textit{orthogonal sum}. All these linear operations can be
obviously extended to \textit{subspaces} of $\cH$. Note that the sum $\overline{\mathfrak{M}}_1 + \overline{\mathfrak{M}}_2$,
or the direct sum $\overline{\mathfrak{M}}_1 \dot + \overline{\mathfrak{M}}_2$ of subspaces is not obligatory a
subspace, but it is true for the orthogonal sum $\overline{\mathfrak{M}}_1 \oplus \overline{\mathfrak{M}}_2$.

We use the symbols $\dom T$, $\ran T$, $\ker T$ for manifolds which are respectively domain, range, and null-subspace
of a linear operator $T$. The closures of two first manifolds are denoted by $\cdom T$, $\cran T$.
The identity operator in a Hilbert space $\cH$ is denoted by $I := I_\cH$. If $\sL$ is a subspace of $\cH$,
the orthogonal projection in $\cH$ onto $\sL$ is denoted by $P_\sL$. By $\sL^\perp$ we denote the subspace which is the
orthogonal complement of $\sL$, which is $\sL^\perp = \cH \ominus \sL $. We use notation $T\uphar \cN$ for
restriction of a linear operator $T$ on the set $\cN\subset\dom T$.

A linear operator $\cA$ in a Hilbert space is called \textit{non-negative} (or \textit{positive})
if $(\cA f,f)\ge 0$ for all $f\in \dom\cA$ and it is called \textit{positive definite} if $(\cA f,f)\ge c \|f\|^2$ for
some $c >0$. We write $\cA \geq 0$ if $\cA$ is a non-negative operator. Then the natural order $\cA \leq \mathcal{C}$ of
two positive (bounded) self-adjoint operators is implied by $\mathcal{C} - \mathcal{A} \geq 0$.

The linear space of bounded operators from the Hilbert space $\cH$ to the Hilbert space $\sH$ is denoted by
$\bB(\cH,\sH)$ and the Banach algebra $\bB(\cH,\cH)$ by $\bB(\cH)$. The set of all bounded
self-adjoint non-negative operators in $\cH$ we denote by $\bB^+(\cH)$. Then \textit{non-singular} operators
$\bB^{+}_{0}(\cH) \subset \bB^{+}(\cH)$ is the subset of  $\bB^+(\cH)$ with $\ker B=\{0\}$ . If $T:\cH\to\sH$ is a closed
linear operator in a Hilbert space $\cH$, then we used to consider the linear manifold $\dom T$ as a Hilbert space with
respect to the
\textit{graph inner product}:
\[
(u,v)_T:=(u,v)_\cH+(Tu, Tv)_\sH \ .
\]

\smallskip

{Now we recall two results, which are established by A.Van Daele. The first result demonstrates some
\textit{pathological} properties of \textit{unbounded} operators. It was inspired by the well-known (and somewhat surprising)
J.~von Neumann theorem \cite{Neumann}, which states that for any unbounded self-adjoint operator $\cA$ there is a unitary
operator $U$ such that $\dom \cA$ and $\dom U^* \cA \, U$ have only the \textit{zero} vector in common.}
\begin{theorem}
\label{daele82-1} {\rm\cite [Theorem 2.2]{Daele}}. Let $T$ be a positive self-adjoint operator in the
Hilbert space $\cH$. If $\ker T=\{0\}$ (non-singular operator), then there exists two densely defined closed
symmetric restrictions $S_1$ and $S_2$ of $T$ such that $\dom S_1\cap\dom S_2=\{0\}$.
\end{theorem}
{In fact one can see from the proof of this theorem that moreover: it is possible to choose the symmetric densely
defined operators  $S_1$ and $S_2$ in such a way that their ranges $\ran S_1$ and $\ran S_2$
are \textit{orthogonal}. This second result was formulated by Van Daele in \cite {Daele2} as a corollary the following
general assertion:}
\begin{theorem} \label{daele82-2}{\rm{\cite[Proposition 3]{Daele2}}}.
Let $B$ be a positive self-adjoint operator in the Hilbert space $\cH$. Suppose that $\ker B =\{0\}$ and $\ran
B\ne \cH$, i.e. the inverse operator $B^{-1}$ is unbounded. Then there exist two linear manifolds $\sM_1$ and $\sM_2$
of $\dom B$ such that: \\
{{\rm{(i)}} $\sM_1 \bot \sM_2$ ,\\
{\rm{(ii)}} the direct sum $\sM_1 \dot + \sM_2$ is dense in $\cH$ , \\
{\rm{(iii)}} the linear manifolds $B \, \sM_1$ and $B \, \sM_2$ are also dense in $\cH$.}
\end{theorem}
\begin{remark} \label{daele82-2bis}
Note that if unbounded operator $T$ in Theorem \ref{daele82-1} is boundedly invertible, then one can put
$B:= T^{-1}$ and apply Theorem \ref{daele82-2} for $B \in \bB^{+}_{0}(\cH)$. This means that one can find two
orthogonal linear manifolds $\sM_1,\,\sM_2 \subset \cH$ and define two symmetric operators $S_1$, $S_2$ with dense
domains $\dom S_1:=B \, \sM_1$, $\dom S_2:=H\, \sM_2$ by
restrictions
\begin{equation*}
S_1:=T\uphar\dom S_1 \ \ , \ \ S_2:=T\uphar\dom S_2  \ .
\end{equation*}
Then by construction of operators $S_1$ and $S_2$ the ranges $\ran S_1 =\sM_1$
and $\ran S_2 =\sM_2$ are orthogonal.
\end{remark}
{The next result is due to K.Schm\"{u}dgen. It was apparently motivated by \cite{Neumann} and by the arguments in
\cite {Daele} and \cite{FW}. In paper \cite{schmud} Schm\"{u}dgen proved  the following assertion.}
\begin{theorem}
\label{SchmTh1} {\rm \cite[Theorem 5.1]{schmud}} Let $H$ be a closed unbounded densely defined linear
operator in the Hilbert space $\cH$. Then there exists an orthogonal projection $P$ such that
\begin{equation}\label{Sch-thm1}
P\cH\cap\dom H=(I-P)\cH\cap\dom H=\{0\} \ .
\end{equation}
\end{theorem}
\begin{remark} \label{daele82-2-SchmTh1}
In fact the statements formulated in Remark \ref{daele82-2bis} and in Theorem \ref{SchmTh1} are equivalent {in the case
when} $\sM_1=\sM$ is a subspace, i.e. $\sM_2=\sM^\perp = \cH \ominus \sM $, and if $H\geq 0$ is a
non-singular, unbounded, self-adjoint operator. \\
Indeed, let $B\in\bB^+_0(\cH)$ with $\ran B\ne\cH$. Then $B$ is invertible and $T:=B^{-1}$ is
unbounded self-adjoint operator $T \geq 0$ with $\dom T = \ran B$. {By Theorem \ref{daele82-2}(i) and by our assumption
that $\sM_1=\sM$, $\sM_2= \cH \ominus \sM $,  Theorem \ref{daele82-2}(iii) yields:}
$(\overline{B\sM^\perp}=\cH) \Leftrightarrow (\forall u \in \sM^\perp \wedge \phi \in \cH ,
(Bu,\phi)=0 \Leftrightarrow \phi = 0)$. Since $B$ is self-adjoint, we have
$(\forall u \in \sM^\perp \wedge \phi \in \cH , (u, B\phi)=0 \Leftrightarrow B\phi \in \sM \cap \ran B)$, and therefore
$(\overline{B\sM^\perp}=\cH \Leftrightarrow \sM \cap \ran B = \{0\})$. The same arguments yield $(\overline{B\sM}=\cH
\Leftrightarrow \sM^\perp \cap \ran B = \{0\})$. Hence, one gets
\begin{multline*}
\overline{B\sM^\perp}=\overline{B\sM}= \cH\iff\sM\cap\ran B=\sM^\perp\cap\ran B=\{0\}\\
\iff \sM\cap\dom T =\sM^\perp\cap\dom T =\{0\} \ ,
\end{multline*}
that gives (\ref{Sch-thm1}) for $\sM = P \cH$ and $T=H$.

On the other hand, let $H\geq 0$ be unbounded, self-adjoint
operator, which is boundedly invertible: $H^{-1}= B$. Then by
Theorem \ref{SchmTh1} there exists an orthogonal projector $P$ such
that
\begin{multline*}
P\cH \cap \dom H=(I-P)\cH \cap \dom H=\{0\}
\iff \sM\cap\ran B =\sM^\perp\cap\ran B =\{0\}\\
\iff \overline{B \sM}=\overline{B \sM^\perp}=\cH  \ ,
\end{multline*}
where $\sM := P \cH$ and  $\sM^\perp = (I-P)\cH $. This coincides
with Remark \ref{daele82-2bis} for $\sM_1 = \sM$ and $\sM_2 =
\sM^\perp$.
\end{remark}
\begin{remark} \label{VD-Schm-Th}
In the present paper we call the statements of Theorems
\ref{daele82-1}-\ref{SchmTh1} and of Remark \ref{daele82-2bis} as
the {\rm{Van Daele--Schm\"{u}dgen Theorem}}. Our aim is to develop a
new systematic approach to treat the pathologies of unbounded
operators, which is motivated by this Theorem.
\end{remark}
{Note that using Theorem \ref{SchmTh1} and the Cayley transformation Schm\"{u}dgen  also proved in \cite{schmud}
an extended version of the Van Daele Theorem \ref{daele82-1}. It is related to the \textit{domain triviality} problem
of the square of symmetric operator. This problem was formulated and studied for the first time in
\cite{Naimark1}, \cite{Naimark2}, \cite{Dix2}, \cite{Chern}.}
\begin{theorem} \label{SchmTh2} {\rm \cite[Theorem 5.2]{schmud}}. For each unbounded
self-adjoint operator $H$ in $\cH$ there exists closed densely defined restrictions of $H$ to symmetric
operators $H_1$ and $H_2$ such that
\begin{equation}\label{Sch-thm2}
\dom H_1\cap\dom H_2=\{0\} \quad \mbox{and} \quad \dom H^2_1=\dom H^2_2=\{0\} \ .
\end{equation}
\end{theorem}
Later, J.R.Brasche and H.Neidhardt \cite{BraNeidh} showed that this result remains true if the condition of self-adjoint
operator $H$ is replaced in Theorem \ref{SchmTh2} by a closed symmetric, but non-self-adjoint operator.

Note also that the first assertion in (\ref{Sch-thm2}) was proved by Van Daele \cite{Daele} under additional assumptions:
$H\geq 0$ and $\ker H=\{0\}$, see Theorem \ref{daele82-1}.

Remark that original proofs of Theorems \ref{daele82-1}, \ref{daele82-2}, and \ref{SchmTh1} are essentially based on
the \textit{spectral decompositions} of self-adjoint operators and the theory of functions (Fourier series,
analytic functions etc). In the present paper {we first elucidate and then we give a new proof of the
Van Daele--Schm\"{u}dgen theorem. The proof includes also a generalisation of this theorem.} To this aim we use only operator
methods. Our approach uses two key ingredients:
\begin{enumerate}
\item
{The classical von Neumann theorem \cite{Neumann} {(see also \cite{Neumann1}, \cite{Neumann2})}, which in particular states
that for any unbounded self-adjoint operator $A$ with a dense domain in $\cH$ there exists a densely defined self-adjoint
operator $B$ such that intersection of their domains is trivial: $\dom A\cap\dom B=\{0\}$.}
\item The notion and properties of a \textit{parallel addition} operation for two bounded non-negative
self-adjoint operators \cite{AD},\cite{AT}.
\end{enumerate}

{As we mentioned above the von Neumann theorem  states in particular that for any \textit{unbounded} self-adjoint
operator $H$ there exists a unitary $U$ such that $\dom H\cap\dom(U^*HU)=\{0\}$. Then setting $J:=2P-I$ for projection
$P$ satisfying (\ref{Sch-thm1}), we obtain as a corollary a refined version of this theorem: there exists a \textit{unitary}
and \textit{self-adjoint} operator $J$ such that $ \ \dom H \cap \dom (J H J) = \{0\}$ }, see Sections \ref{triv-intersec}
and \ref{Appl}.

Our arguments allow to obtain more details about properties of restrictions of self-adjoint operators treated in Theorems
\ref{daele82-1}-\ref{SchmTh2} and to revise the Van Daele--Schm\"{u}dgen and the Brasche--Neidhardt
theorems, see Section \ref{VD-Sch-B-N}.
{We note also that the von Neumann theorem \cite{Neumann}-\cite{Neumann2} and Schm\"{u}dgen's result \cite{schmud-1}
are related to results in \cite{NeidhZag1}, \cite{NeidhZag2} about another kind of pathological properties of
operators unbounded from above \textit{and} from below. These papers solved the problem of existence of densely defined
symmetric semi-bounded restrictions to \textit{stability domains} of initially unbounded from \textit{below} symmetric
operators. The same theorems together with the operator \textit{parallel addition} play essential role in \cite{AHS2} in
order to construct counterexamples to some statements in \cite{KrO} related to the $Q$-functions of Hermitian contractions.}

Here is a brief review of contents of the paper.
{In Section \ref{Prel} we recall some basic facts of the operator theory indispensable for formulations and proofs
of our main results. They are: the operator ranges, the concept of parallel addition, the Kre\u\i n shorted operators, the
self-adjoint extensions of non-negative operators, and few relevant fundamental statements like von Neumann's and
Douglas' theorems.}

{Section \ref{main-results} collects our main results. We start by Section \ref{triv-intersec}, where different
characterisations for trivial intersections of operator ranges with subspaces are presented. This preparation is aimed to
describe then essential steps of our approach.}

Our key statement (Theorem \ref{SchmTh333}) is that for a given
$A\in\bB_{0}^+(\cH)$ with $\ran A\ne \cH$, we can find a
{\textit{continuum} set of different subspaces} $\sM \subset \cH$
satisfying
\begin{equation}\label{zeroint111a}
\sM\cap \ran A^{1/2}=\sM^\perp\cap\ran A^{1/2}=\{0\} \ .
\end{equation}

{In Theorem \ref{monot} we show the existence of \textit{increasing} (\textit{decreasing}) chains of subspaces
possessing the trivial intersection property \eqref{zeroint111a}.} To this aim we use the \textit{lifting} of
operator $A$. It is defined as a representation of $A$ generated by orthogonal projection $P_\sM: \cH \rightarrow \sM $,
which has the form
\begin{equation}\label{lift-Intr}
A=T^{1/2}P_\sM T^{1/2} \ .
\end{equation}
Here $T\in\bB^+_0(\cH)$ is the sum $T=A+B$, where $B\in\bB_{0}^+(\cH)$ with $\ran B^{1/2}\cap\ran A^{1/2}=\{0\}$. Note
that by virtue of trivial intersection of $\ran B^{1/2}$ and $\ran A^{1/2}$, the subspaces $\sM$ and $\sM^\perp$ have trivial
intersections with $\ran T^{1/2}$ :
\begin{equation}\label{T-intersect-M}
\sM\cap\ran T^{1/2}=\sM^\perp\cap\ran T^{1/2}=\{0\}.
\end{equation}

{In Section \ref{LiftOper} we study the existence of the lifting in the form \eqref{lift-Intr} with \eqref{T-intersect-M},
when the subset $\sM$ possessing \eqref{zeroint111a} is given. Then conditions on the entries of $A\in\bB^+_0(\cH)$ in its
block-operator matrix representation with respect to decomposition $\cH=\sM\oplus\sM^\perp$ are found. We give examples
that \textit{not all} subspaces $\sM$ possessing \eqref{zeroint111a} can be constructed applying a \textit{general} form of
the {operator lifting} (\ref{lift-Intr}). This indicates that our method is not exhaustive.
It also means that the problem of construction of \textit{all} subspaces $\sM$ verifying \eqref{zeroint111a} for a given
operator $A$ is \textit{open}.}

Nevertheless, our method of the operator \textit{lifting} allows to obtain more detailed information about the
\textit{hierarchy} of possible subspaces $\sM$ and {to establish a number of new results about it.} In particular, we prove
that for a given $A\in\bB_{0}^+(\cH)$ with $\ran A\ne \cH$ there exists a one-parameter family of these subspaces, see
Theorem \ref{SchmTh333} and Proposition \ref{11}, as well as some increasing (decreasing) infinite chains of subspaces
$\sM$ with the property \eqref{zeroint111a}, see Theorem \ref{monot} and Corollary \ref{nov1}.

In Theorems \ref{111}, \ref{appl1} of the next Section \ref{Appl} we revise the Schm\"{u}gen result (Theorem \ref{SchmTh1}).
{Moreover, in Theorem \ref{now11} we construct decreasing/increasing families
(in the sense of associated closed quadratic forms) of \textit{pairs} of non-negative self-adjoint operators with trivial
interactions of their form-domains with domain of a given unbounded non-negative self-adjoint operator.
Then we investigate the limiting behaviour of their resolvents and of the corresponding one-parameter semigroups.}

{These results allow to scrutinise in Section \ref{VD-Sch-B-N} the \textit{triviality domain problem} for products/powers of
\textit{unbounded} operators, cf Theorem \ref{SchmTh2}. We propose a systematic method for construction of examples of
pairs operators $\langle B, \wt B \rangle $ consisting of closed densely defined symmetric operator $B$ and its
symmetric/self-adjoint extension $\wt B$, such that $\dom(\wt B^*B)=\{0\}$. This gives abstract examples of symmetric
operators $B$ with trivial squares and allows us to refine the Van Daele--Schm\"{u}dgen and the Brasche--Neidhardt theorems.
Under certain additional conditions we show in Theorems \ref{novsch} and \ref{brne} that the products in different
order, i.e., operators $B\wt B$ and $\wt B B$ are densely defined  and we describe their Friedrichs and Kre\u\i n
self-adjoint extensions}.

\section{Preliminaries}\label{Prel}
\subsection{Operator ranges}
{Following \cite{FW} we call linear manifold $\cR$ in a Hilbert space $\cH$ an \textit{operator range},
if it is the range of \textit{some} bounded linear operator on $\cH$. Note that even for bounded operators
the operator ranges possess certain special features that distinguish them from \textit{arbitrary} linear manifolds
since their properties may be more pathological.}

Clearly, if an operator range $\cR$ is unclosed and dense in $\cH$, then it is a domain of a
non-negative self-adjoint unbounded operator in $\cH$. Indeed, if $\cR=\ran\cA$, $\cA\in\bB(\cH)$, then $\cR=\ran |\cA^*|$,
where $|\cA^*|:=(\cA\cA^*)^{1/2}$ is non-negative self-adjoint bounded operator. Since $\cR$  is dense in $\cH$ we get
$\ker |\cA^*|=\{0\}$. The inequality $\cR\ne \cH$ yields that the operator $T=|\cA^*|^{-1}$ is unbounded non-negative
self-adjoint operator and $\dom T=\cR.$ Conversely, if $T$ is a non-negative unbounded closed and densely defined linear
operator, then $\dom T=\dom|T|$, for $|T|=(T^* T)^{1/2}$. Consequently, one obtains
\[
\dom T=\ran (|T|+I)^{-1} \ .
\]
This means that $\dom T$ is an operator range. Various characterizations of operator ranges can be found in \cite{FW}.

\subsection{The Douglas theorem}

\begin{theorem}{\rm \cite{Doug}} \label{Douglas}
For every $A, B\in \bB(\cH)$ the following statements are equivalent:
\begin{enumerate}
\def\labelenumi{\rm (\roman{enumi})}
\item $\ran A\subset \ran B$;
\item $A=BC$ for some $C\in \bB(\cH)$;
\item $AA^*\le \lambda BB^*$ for some $\lambda \ge 0$.
\end{enumerate}
{Moreover, there is a unique operator $C$ satisfying $\ran C\subset\cran B^*$, in which case $\ker C=\ker A$.}
\end{theorem}
The next relations follow from Theorem \ref{Douglas} (see \cite{FW}, Sect.4):
\begin{equation}
\label{doug11} \left(\sum\limits_{j=1}^n F_j\right)^{1/2}=\ran
F^{1/2}_1+\ldots+\ran F^{1/2}_n,\; \{F_j\}_{j=1}^n\subset\bB^+(\cH),
\end{equation}
\begin{equation}
\label{usef} \ran \left(F^{1/2}MF^{1/2}\right)^{1/2}=F^{1/2}\ran
M^{1/2}, \; F,M\in\bB^+(\cH).
\end{equation}
This yields, in particular, that if $T_1, \ldots, T_j$ are closed
and densely defined linear operators in $\cH$, then the linear
manifold
\[
\dom T_1+\ldots+ \dom T_n
\]
is domain of a closed linear operator.

\subsection{The von Neumann theorem} In paper \cite{Neumann} (see also \cite{Neumann1}, \cite{Neumann2})
John von Neumann established the following fundamental result:
\begin{theorem} \label{NeumTh1}
For any unbounded self-adjoint operator $H$ in a Hilbert space there
exists a unitary operator $U$ with the property
\[
\dom H \cap \dom (U^*HU)=\{0\}.
\]
\end{theorem}
Special examples of self-adjoint operators $A$ and $B$ with $\dom A\cap\dom B=\{0\}$ one can find in \cite{CNZ},
\cite{Daele}, \cite{Kosaki}.

In terms of operator ranges the statement of Theorem \ref{NeumTh1} takes the following form:
\begin{theorem} \label{NeumTh2}
If $\cR$ is a nonclosed {and dense operator range} in a Hilbert space $\cH$, then there is a unitary operator $U$
on $\cH$ such that
\begin{equation}
\label{fw11} \cR\cap U\cR=\{0\}.
\end{equation}
\end{theorem}
\begin{corollary}
\label{coro1} If $\cR$ is a nonclosed operator range in a Hilbert
space, then there exists a continuous one-parameter unitary group $\{U_t\}_{t\in \mathbb{R}}$ such that
$U_s \cR \cap U_t \cR =\{0\}$ for $s \ne t$.
\end{corollary}
The proof of Theorem \ref{NeumTh2} and Corollary \ref{coro1} can be found in e.g. \cite{Dix} and \cite{FW}.

\subsection{The parallel sum of operators.}
Let $F$ and $G$ be two bounded non-negative operators on $\cH$. The
\textit{parallel sum} $F:G$ of $F$ and $G$ is defined by the quadratic form:
\[
\left((F:G)h,h\right):=\inf_{f,g \in \cH}\left\{\,\left(Ff,f\right)+\left(Gg,g\right):\,
   h=f+g \,\right\} \ ,
\]
see \cite{AD}, \cite{FW}, \cite{KA}. One can establish for $F:G$  the following equivalent
definition
\[
F:G=s-\lim\limits_{\varepsilon\downarrow 0}\,
F\left(F+G+\varepsilon I\right)^{-1}G \ ,
\]
see \cite{AT}, \cite{PSh}. Then for \textit{positive definite} bounded self-adjoint operators $F$ and $G$ we obtain
\[
F:G=(F^{-1}+G^{-1})^{-1} \ .
\]


Since $F\le F+G$ and $G\le F+G$, one gets
\begin{equation}
\label{fg2}
 F=(F+G)^{1/2}M(F+G)^{1/2},\quad  G=(F+G)^{1/2}(I-M)(F+G)^{1/2}
\end{equation}
for some non-negative contraction $M$ on $\cH$ with $\ran M\subset\cran(F+G)$. This yields yet another description
of the parallel sum $F:G$.
\begin{lemma} {\rm \cite{Ar2}}
\label{yu1} Suppose $F, G\in \bB^+(\cH)$ and let $M$ be as in
\eqref{fg2}. Then
\[
 F:G=(F+G)^{1/2}(M-M^2)(F+G)^{1/2}.
\]
\end{lemma}
Using \eqref{doug11} and \eqref{fg2} one obtains the equalities
$$\ran F^{1/2}=(F+G)^{1/2}\ran M^{1/2},\; \ran
G^{1/2}=(F+G)^{1/2}\ran (I-M)^{1/2}.$$
Since
$$
\ran M^{1/2}\cap\ran (I-M)^{1/2}=\ran (M-M^2)^{1/2},
$$
the next proposition is an immediate consequence of Lemma \ref{yu1}, cf. \cite{FW}, \cite{PSh}.
\begin{proposition}
\label{root} 1) $\ran (F:G)^{1/2}=\ran F^{1/2}\cap\ran G^{1/2}$.

2) The following statements are equivalent:
\begin{enumerate}
\def\labelenumi{\rm (\roman{enumi})}
\item
 $F:G=0$;
 \item
the operator $M$ in \eqref{fg2} is an orthogonal projection in
$\cran(F+G)$;
\item $\ran F^{1/2}\cap\ran G^{1/2}=\{0\}$.
\end{enumerate}
\end{proposition}

\subsection{The Kre\u\i n shorted operator}
For a given non-negative bounded operator $B$ on the Hilbert space $\cH$ and for any subspace $\cK\subset \cH\;$
M.G.~Kre\u{\i}n  defined in \cite{Kr1} the operator
\[
B_{\cK}:=\max\left\{\,Z\in \bB(\cH):\,
    0\le Z\le B, \, {\ran}Z\subseteq{\cK}\,\right\}.
\]
Then equivalent definition of $B_{\cK}$ has the following quadratic-form expression:
\begin{equation}
\label{Sh1}
 \left(B_{\cK}f, f\right):=\inf\limits_{\f\in \cK^\perp}\left\{\left(B(f + \varphi),f +
 \varphi\right)\right\},
\quad  f\in\cH.
\end{equation}
Here $\cK^\perp:=\cH\ominus{\cK}$. The operator $B_{\cK}$ is called the \textit{shorted operator} of $B$,
see \cite{And, AT}. Let the subspace $\Omega_{\cK}$ be defined by
\[
 \Omega_{\cK}:=\{\,f\in \cran B:\,B^{1/2}f\in {\cK}\,\}=\cran B\ominus
B^{1/2}\cK^\perp.
\]
Then the shorted operator $B_{\cK}$ gets the form $B_{\cK}=B^{1/2}P_{\Omega_{\cK}}B^{1/2}$ and
\[
 {\ran}B_{\cK}^{1/2}= \cK \cap{\ran} B^{1/2} \ ,
\]
see \cite{Kr1}. In particular, this implies the equivalence:
\begin{equation}\label{nol}
B_{\cK}=0 \iff \cK\cap\ran B^{1/2}=\{0\} \ .
\end{equation}

{Note that with respect to orthogonal decomposition $\cH = \cK \oplus \cK^\perp $ a bounded self-adjoint operator
$B$ has the block-matrix form:}
\[
B=\begin{bmatrix}B_{11}&B_{12}\cr B^*_{12}&B_{22}
\end{bmatrix}:\begin{array}{l}\cK\\\oplus\\\cK^\perp \end{array}\to
\begin{array}{l}\cK\\\oplus\\\cK^\perp \end{array},
\]
where $B_{11}\in\bB(\cK)$, $B_{22}\in\bB(\cK^\perp)$,
$B_{12}\in\bB(\cK^\perp,\cK)$.
It is well-known (see e.g. ) Recall that the operator $B$ is non-negative if and only if \cite{KrO}
\begin{equation}\label{POZ}
B_{22}\ge 0,\; \ran B^*_{12}\subset\ran B^{1/2}_{22},\,\; B_{11}\ge
\left(B^{[-1/2]}_{22}B^*_{12}\right)^*\left(B^{[-1/2]}_{22}B^*_{12}\right) \ .
\end{equation}
{Here $B_{22}^{[-1/2]}: = (B_{22}^{1/2}\upharpoonright \cran B_{22})^{-1}.$}
Then operator $B_\cK$ is given by the block matrix
\begin{equation}\label{posshor1}
B_\cK=\begin{bmatrix}B_{11}-\left(B^{[-1/2]}_{22}B^*_{12}\right)^*\left(B^{[-1/2]}_{22}B^*_{12}\right)&0\cr
0&0\end{bmatrix}.
\end{equation}

Conditions \eqref{POZ} imply that the block operator matrix
 \[
B=\begin{bmatrix}B_{11}&B_{12}\cr B^*_{12}&B_{22}
\end{bmatrix}:\begin{array}{l}\cK\\\oplus\\\cK^\perp \end{array}\to
\begin{array}{l}\cK\\\oplus\\\cK^\perp \end{array}
\]
is non-negative if and only if it takes the form \cite{Shmul}
\begin{equation}\label{shmul1}
B=\begin{bmatrix}B_{11}&B^{1/2}_{11}\Gamma B^{1/2}_{22}\cr
B^{1/2}_{22}\Gamma^* B^{1/2}_{11}&B_{22}
\end{bmatrix}
\end{equation}
where $\Gamma:\cran B_{22}\to\cran B_{11}$ is a contraction. Then from (\ref{posshor1}) it follows that
\begin{equation}
\label{shmul2} B_\cK=\begin{bmatrix}
B^{1/2}_{11}(I-\Gamma^*\Gamma)B^{1/2}_{11}&0\cr 0&0\end{bmatrix},\;
B_{\cK^\perp}=\begin{bmatrix} 0&0\cr
0&B^{1/2}_{22}(I-\Gamma\Gamma^*)B^{1/2}_{22}\end{bmatrix}.
\end{equation}

\subsection{Friedrichs and Kre\u\i n self-adjoint extensions}\label{FrKR}
Let $\cH$ be a Hilbert space and let $\cA$ be a densely defined
closed, symmetric, and non-negative operator. Denote by ${\cA}^*$
the adjoint to $\cA$. Recall that the operator $\cA$ admits at least
one non-negative self-adjoint extension $\cA_{\rm F}$ (called the
\textit{Friedrichs}, or \textit{"hard"} extension \cite{Kr1}), which is
defined as follows. Denote by $\mathfrak{a}[\cdot,\cdot]$ the
closure of corresponding to $\cA$ sesquilinear form
\[
\mathfrak{a}[f,g]=(\cA f,g),\; f,g\in\dom(\cA) \ ,
\]
and let $\cD[\mathfrak{a}]$ be domain of this closure. According to the \textit{first representation theorem} \cite{Ka}
there exists a unique non-negative self-adjoint operator $\cA_F$ associated with $\mathfrak{a}[\cdot,\cdot]$,
i.e.,
\[
(\cA_{\rm F} h ,\psi)=\mathfrak{a}[h,\psi],\; \psi\in\cD[\mathfrak{a}],\; h\in\dom
\cA_{\rm F} \ .
\]
One clearly gets that $\cA\subset \cA_{\rm F}\subset\cA^*$ and that $\dom \cA_{\rm F}=\cD[\mathfrak{a}]\cap\dom \cA^*.$
Moreover, by the \textit{second representation theorem} \cite{Ka} the following equalities
\[
\cD[\mathfrak{a}]=\dom\cA^{1/2}_{\rm
F}\quad\mbox{and}\quad\mathfrak{a}[\phi,\psi]=(\cA^{1/2}_{\rm F}\phi,\cA^{1/2}_{\rm F}\psi),\ \
\phi, \psi\in \cD[\mathfrak{a}] \ ,
\]
also hold.

In \cite{Kr1} M.G.~Kre\u{\i}n discovered one more non-negative self-adjoint extension $\cA_{\rm K}$ of $\cA$.
It has the extremal property to be a \textit{minimal}, whereas the Friedrichs extension $\cA_{\rm F}$ is the \textit{maximal}
{(in the sense of the corresponding associated closed quadratic forms)} {among \textit{all} other non-negative self-adjoint
extensions $\mathcal{C}$ of $\cA$ : $\cA_{\rm K} \leq \mathcal{C} \leq \cA_{\rm F}$.} These inequalities are
equivalent to inequalities for resolvents :
\[
(\cA_{\rm F}+aI)^{-1}\le (\cC+a I)^{-1}\le (\cA_{\rm K}+a I)^{-1},\; a>0 \ ,
\]
see \cite{Ka}, \cite{Kr1}. The extension $\cA_{\rm K}$ is called the \textit{Kre\u{\i}n extension} of $\cA$. If $\cA$ is
a positive-definite symmetric operator, then the subspace $\ker\cA^*$ is nontrivial and one gets:
\[
\dom\cA_{\rm K}=\dom \cA\dot+\ker\cA^* \ ,
\]
see \cite{Kr1}, whereas
\begin{equation*}
\dom\cA_{\rm F}=\dom \cA \dot+ (\cA_{\rm F})^{-1}\ker\cA^* \ .
\end{equation*}

Let $L_1$ and $L_2$ be closed linear operators defined in a Hilbert
space $\cH$, taking values in a Hilbert space $\sH$, such that $L_2$
is extension of $L_1$:
\begin{equation} \label{cond}
L_1\subset L_2.
\end{equation}
Then operators $L^*_1L_1 $ and $L^*_2L_2$ are self-adjoint and  non-negative. Since $L_2^*\subset L^*_1$, the following
relations are valid:
\[
\dom(L^*_2L_1)=\dom(L^*_1L_1)\cap\dom(L^*_2L_2)={\dom L_1\cap\dom(L^*_2L_2).}
\]
Suppose that
\begin{equation}
\label{densedom}\dom(L^*_1L_1)\cap\dom(L^*_2L_2)\ne\{0\} \ .
\end{equation}
Then operator $\cA$ defined as follows:
\begin{equation}
\label{Factor}
\begin{array}{l}
\cA f:=L^*_2L_1 f, \ f\in\dom\cA, \ \ {\rm{for}} \ \ \dom\cA:=\dom(L^*_2L_1)\ ,
\end{array}
\end{equation}
is closed and  symmetric. Since  $(\cA f,f)=||L_1f||^2\ge 0$ for all
$f\in\dom \cA$, the operator $\cA$ is non-negative. This kind of operators $\cA$ we call the
\textit{operators in divergence form} \cite{ArlKov_2011}. The next assertions are established in \cite{ArlKov_2011}.
\begin{theorem}\label{Sim} {\rm\cite [Theorem 3.1]{ArlKov_2011}}. Let $L_1, L_2:
\cH\to\sH$ be closed and densely defined operators, satisfying
condition \eqref{cond}.  If the operator $\cA=L^*_2L_1$ is densely
defined and its adjoint is given by
\[
\cA^*=L^*_1L_2,
\]
 then
\begin{enumerate}
\item the Friedrichs extension of $\cA$ is given by the
operator $L^*_1L_1$, i.e.,
\[
\dom \cA_{\rm F}=\{f\in \dom L_1: L_1f\in\dom L_1^*\},\; \cA_{\rm
F}f=L^*_1L_1f, \; f\in\dom\cA_{\rm F},
\]
\item
\[
\dom \cA^{1/2}_{\rm F}=\dom  L_1,\; (\cA^{1/2}_{\rm
F}u,\cA^{1/2}_{\rm F}v)=(L_1u, L_1v),\; u,v\in\dom L_1,
\]
\item  the Kre\u{\i}n extension of $\cA$ is the operator  $\cA_{\rm K} =L_2^*P_{\cran
L_1}L_2$, i.e.,
\[
\begin{array}{l}
\dom \cA_{\rm K}=\{f\in \dom L_2:P_{\cran L_1}L_2f\in\dom L^*_2\},\\
\cA_{\rm K} f=L^*_2P_{\cran L_1}L_2f,\; f\in\dom \cA_{\rm K},
\end{array}
\]
and
\[
\dom \cA^{1/2}_{\rm K}=\dom L_2,\; (\cA^{1/2}_{\rm
K}u,\cA^{1/2}_{\rm K}v)=(P_{\cran L_1}L_2u,P_{\cran L_1}L_2v),\;
u,v\in\dom L_2,
\]
\item the Friedrichs and the Kre\u{\i}n extensions of $\cA$
are transversal, i.e., one gets for their domains:
\[
\dom \cA_{\rm F}+\dom\cA_{\rm K}=\dom \cA^*.
\]
\end{enumerate}
\end{theorem}
\section{Main results}\label{main-results}

This section collects our main results. They are based on some new ideas and our lines reasoning improve the results
in \cite{Daele}, \cite{schmud}, \cite{BraNeidh}. We give new proofs and generalise the Van Daele--Schm\"{u}dgen
Theorems \ref{daele82-1},\ref{daele82-2},\ref{SchmTh1},\ref{SchmTh2} and the Brasche--Neidhardt assertion \cite{BraNeidh}.

Our observations also lead to certain new applications, see Section \ref{Appl}.
\subsection{Trivial intersections of operator ranges with subspaces}\label{triv-intersec}
{We start this section by a useful refinement of the von Neumann Theorem \ref{NeumTh1}, which we reformulated in
Theorem \ref{NeumTh2} in terms of ranges.}

{A bounded linear operator $J$ on a Hilbert space $\cH$ is self-adjoint and unitary operator
if and only if one has:
\[
J=J^*=J^{-1} \ .
\]
}
{Such operator is often called \textit{fundamental symmetry}, or \textit{signature operator} \cite{Bognar}.
Note that $J$ is a fundamental symmetry operator if and only if
\[
J = 2P-I,
\]
where $P$ is an orthogonal projection in $\cH$.}
{\begin{proposition} \label{dobavl}
Let $\cR$ be a non-closed linear manifold in a Hilbert space $\cH$. Then the following assertions
are equivalent:
\begin{enumerate}
\def\labelenumi{\rm (\roman{enumi})}
\item There exists in $\cH$ an orthogonal projection $P$ such that
\[
\ran P\cap\cR=\{0\} \ \ \ {\rm{and}}  \ \ \ \ran (I-P)\cap\cR=\{0\} \ .
\]
\item There exists in $\cH$ a fundamental symmetry $J$  such that
\[
J\cR\cap\cR=\{0\} \ .
\]
\end{enumerate}
\end{proposition}
}
\begin{proof}
{(i) $\Rightarrow$ (ii). Set $J:=2P-I$. Let $f\in \cR$ and suppose $Jf\in\cR$. Then $2Pf\in\cR$. But
$\ran P\cap\cR=\{0\}$. Hence $f\in\ran (I-P).$ Since $\ran (I-P)\cap\cR=\{0\}$, we obtain $f=0$, i.e.,
the statement (ii) holds.}

{(ii) $\Rightarrow$ (i). Let $P:=(I+J)/2$. Then $P$ is orthogonal projection in $\cH$.
Suppose that $f\in(\ran P\cap\cR)$. Then $Jf=Pf=f\in\cR$ and by virtue of (ii) one obtains $f=0$. A similar argument
is valid for $f\in(\ran (I-P)\cap\cR=\{0\})$.}
\end{proof}

\begin{proposition}
\label{shm01} Let $\cH$ be a Hilbert space. Let $A\in\bB^{+}_{0}(\cH)$ and $\ran A\ne \cH.$
\begin{enumerate}
\item
Let $\sM$ be a subspace in $\cH$ and $P_\sM$ be orthogonal projection on $\sM$. We define the operator
$A_1:=A^{1/2}P_\sM A^{1/2}$.
\begin{enumerate}
\item Then one gets
\[
\begin{array}{l}\sM^\perp\cap\ran A^{1/2}=\{0\}\iff\ker A_1=\{0\}\iff {\overline{A^{1/2}\sM}=\cH},\\
 \sM\cap\ran A^{1/2}=\{0\}\iff \ran
A^{1/2}_1\cap\ran A=\{0\}.
\end{array}
\]
Hence, the
 following statements are equivalent:
\begin{enumerate}
\def\labelenumi{\rm (\roman{enumi})}
\item $\ran A^{1/2}_1\cap\ran A=\{0\}$ and $\ker A_1=\{0\},$
\item $\sM\cap\ran A^{1/2}=\sM^\perp\cap\ran A^{1/2}=\{0\},$
\item  {the linear manifolds $ A^{1/2}\sM$ and $A^{1/2}\sM^\perp$ are dense in $\cH$.
}
\end{enumerate}
\item If $\ker A_1=\{0\}$, then
\begin{equation}
\label{restr} ||A^{-1/2}_1h||=||A^{-1/2}h||\;\mbox{for all}\;
h\in\ran A^{1/2}_1.
\end{equation}
\end{enumerate}
\item
If $A,A_1\in\bB^+(\cH)$, $\ran A^{1/2}_1\subset\ran A^{1/2}$ and if
\eqref{restr} holds true, then $A_1=A^{1/2}PA^{1/2}$, where $P$ is
an orthogonal projection in $\cH$.
\end{enumerate}
\end{proposition}
\begin{proof}
(1) By definition of $A_1$ and by the Douglas Theorem \ref{Douglas} we have $\ran A^{1/2}_1=A^{1/2}\sM$.
It follows then that
$$\ran A^{1/2}_1\cap\ran A=\{0\} \iff \sM\cap\ran A^{1/2}=\{0\}.$$
It is also clear that
\[
\ker A_1=\{0\}\iff \cran A_1=\cH\iff \sM^\perp\cap\ran A^{1/2} = \{0\}.
\]
The equality: $||A^{1/2}_1f||^2=||P_\sM A^{1/2}f||^2$ for all $f\in\cH$,
implies that there is an isometry $V:\sM\to \cH,$ $\ran V=\cH$ such
that $A^{1/2}_1h=VP_\sM A^{1/2}h$, $h\in\cH$. Hence
\[
A^{1/2}_1=A^{1/2}V^*,\; A^{-1/2}h=V^* A^{-1/2}_1h,\; h\in\ran
A^{1/2}_1,
\]
where $V^*:\cH\to\sM$, $\ran V^*=\sM$ and $V^*$ is isometry.

For the proof of the statement (2) we refer to \cite{ArlKov_2013}.
\end{proof}
\begin{proposition} \label{mars1}
Let $A\in\bB_{0}^+(\cH)$ and $\ran A\ne \cH.$
Let $P_1$ and $P_2$ be two orthogonal projections in $\cH$ such that
\begin{equation}
\label{dvoi}
\ran P_k\cap \ran A^{1/2}=\ran(I-P_k)\cap\ran A^{1/2}=\{0\},\; k=1,2
\end{equation}
If we define
\[
A_1:=A^{1/2}P_1 A^{1/2},\;A_2:=A^{1/2}_1P_2 A^{1/2}_1 ,
\]
then
\[
A_2=A^{1/2}P_{12}A^{1/2},
\]
where $P_{12}$ is an orthogonal projection such that
\begin{equation}
\label{dvoi2}
\ran P_{12}\cap \ran A^{1/2}=\ran(I-P_{12})\cap\ran A^{1/2}=\{0\}.
\end{equation}
\end{proposition}
\begin{proof}
We have $A^{1/2}_1=V_1P_1 A^{1/2}$, $A^{1/2}_2=V_2P_2 A^{1/2}_1,$ where $V_1:\ran P_1\to\cH,$ $V_2:\ran P_2\to\cH$
are isometries. Then
\[
A^{1/2}_1= A^{1/2}V^*_1,\;A^{1/2}_2=
A^{1/2}_1V^*_2,\;V^*_k:\cH\to\ran P_k,\; k=1,2.
\]
It follows that
\[
A^{1/2}_2=A^{1/2}V^*_2V^*_1.
\]
The operator $V:=V^*_2V^*_1$ is isometry,
$$\ran V=V^*_2 \ \{\ran V^*_1\}=V^*_2 \ \{\ran P_1\} \subset \ran P_2.$$
Let $P_{12}:=P_{\ran V} $ be orthogonal projection on $\ran V$. Then
\[
A_2=A^{1/2}VV^*A^{1/2}=A^{1/2}P_{12}A^{1/2}.
\]
Using \eqref{dvoi} and Proposition \ref{shm01}, we obtain that $\ker
A_2=\{0\}$, $\ran A^{1/2}_2\cap\ran A_1=\{0\}$, and $\ran
A^{1/2}_1\cap\ran A=\{0\}$. Due to inclusion $\ran
A^{1/2}_2\subset\ran A^{1/2}_1$, one gets $\ran A^{1/2}_2\cap\ran
A=\{0\}$. Then application of Proposition \ref{shm01} leads to
equalities \eqref{dvoi2}.
\end{proof}

\begin{proposition} \label{mars2}
Let $A\in\bB_{0}^+(\cH)$, $\ran A\ne \cH$ and
let $P_1$ and $P_2$ be two orthogonal projections in $\cH$ such that $P_1\le P_2$ and \eqref{dvoi} holds.
Define
\[
A_1:=A^{1/2}P_1A^{1/2}, \  A_2:=A^{1/2}P_2A^{1/2}.
\]
Then
\[
A_1=A^{1/2}_2PA^{1/2}_2,
\]
where $P$ is an orthogonal projection such that
\[
\ran (I-P)\cap\ran A^{1/2}_2=\{0\}.
\]
\end{proposition}
\begin{proof}
It is clear that
\[
A^{1/2}_1=V_1P_1A^{1/2},\; A^{1/2}_2=V_2P_2A^{1/2},
\]
where the operator $V_k:\ran P_k\to \cH$ is isometry, $k=1,2$. Then
\[
A^{1/2}_1=A^{1/2}_2V_2V^*_1
\]
and $V:=V_2V^*_1$ is isometry with $\ran V=V_2 \{\ran P_1\}$. Hence
\[
A_1=A^{1/2}_2PA^{1/2}_2,
\]
where $P := P_{\ran V}$ is orthogonal projection on $\ran V$. Since $\ker A_1=\{0\}$, we obtain the last statement,
$\ran (I-P)\cap\ran A^{1/2}_2=\{0\}$, of the theorem.
\end{proof}

\begin{proposition}
\label{shm1} Let $\cH$ be a Hilbert space.

1) Suppose that
\begin{equation}
\label{cond1}
\begin{array}{l}
 F,\; G\in\bB^+(\cH),\; \ker F=\ker G=\{0\},\\
 F:G=0(\iff\;\ran F^{1/2}\cap\ran G^{1/2}=\{0\}).
\end{array}
\end{equation}
Then there is a subspace $\sM$ in $\cH$ satisfying
\begin{equation}
\label{zeroint}
\sM\cap\ran(F+G)^{1/2}= \sM^\perp\cap\ran(F+G)^{1/2}=\{0\}.
\end{equation}
2) If a subspace $\sM$ in $\cH$ is such that $\dim \sM=\dim\sM^\perp=\infty$, then there is a pair of linear operators
$F$ and $G$ satisfying \eqref{cond1} and \eqref{zeroint} .
\end{proposition}
\begin{proof}
By virtue of Proposition \ref{root} and equalities \eqref{fg2} we have
\[
F=(F+G)^{1/2}P(F+G)^{1/2},\quad  G=(F+G)^{1/2}(I-P)(F+G)^{1/2},
\]
where $P$ is an orthogonal projection in $\cH$. Put $\sM:=\ran P$. Then, since $\ker F=\ker G=\{0\}$, we obtain
\eqref{zeroint}.

Conversely, suppose that $\sM$ is a subspace in $\cH$ such that $\dim \sM=\dim\sM^\perp=\infty$.
Then by \cite{ArlKov_2013} one can find operator $X\in\bB(\cH)$, $X\ge 0$, such that $\ker X=\{0\}$ and
\begin{equation}
\label{zerint2} \sM\cap \ran X^{1/2}=\{0\},\;\sM^\perp\cap \ran
X^{1/2}=\{0\}.
\end{equation}
Now we follow the line of reasoning close to constructions in \cite[Section 5]{ArlKov_2013}. To this end note that by
Theorems \ref{NeumTh1} and \ref{NeumTh2} one can find in the subspace $\sM$ two non-negative self-adjoint operators $W$ and
$V$ from $\bB(\sM)$ possessing the properties
\[
\begin{array}{l}
\cran W=\cran V=\sM,\; \ran W\cap\ran V=\{0\},\\
0\le W\le I_\sM,\;\ker W=\{0\},\;0\le V\le I_\sM,\;\ker V=\{0\}.
\end{array}
\]
Let us replace $V$ by operator $U=V\Phi$, where $\Phi$ is a unitary operator from $\sM^\perp$ onto $\sM$ {and
taking into account $\cH=\sM \oplus \sM^\perp$, define}
\[
X: = \begin{bmatrix}
W^2&WU\cr U^*W&U^*U  \end{bmatrix}.
\]
Let us show that
\[
\ker X=\{0\},\; X_\sM=0,\; X_{\sM^\perp}=0.
\]
Set $f=\begin{bmatrix}f_1\cr f_2 \end{bmatrix},$ where $
f_1\in\sM,\; f_2\in\sM^\perp$. Then
\begin{equation}
\label{quadrform} (Xf,f)=||Wf_1+Uf_2||^2.
\end{equation}
It follows that
\[
Xf=0\iff Wf_1+Uf_2=0.
\]
Since $\ran W\cap\ran U=\{0\},$ $\ker W=\{0\}$, $\ker U=\{0\}$, we get $f_1=0,$ $f_2=0$. From \eqref{quadrform} and
relations $\cran W=\cran U=\sM$ we get the equalities
\[
\inf\limits_{\f\in \sM^\perp}(X(f-\f), f-\f)=0,\;
\inf\limits_{\psi\in \sM}(X(f-\psi), f-\psi)=0.
\]
Equality \eqref{Sh1} now implies $X_\sM=0$ and $X_{\sM^\perp}=0.$
Applying \eqref{nol} we obtain \eqref{zerint2}.

Now set
\[
F=X^{1/2}P_\sM X^{1/2},\; G=X^{1/2}(I- P_\sM) X^{1/2}.
\]
Then by construction $\ker F=\ker G=\{0\}$, $\ran F^{1/2}= X^{1/2}\sM,$ $\ran G^{1/2}=X^{1/2}\sM^\perp$.
Hence
$$\ran F^{1/2}\cap\ran G^{1/2}=\{0\},$$
therefore by Proposition \ref{root}  one obtains $F:G=0$. Since $F +
G = X$, this proves the statement 2). Hence, the proof of the
proposition is completed.
\end{proof}
\begin{corollary}
\label{VAZZN} Let $X\in \bB_{0}^+(\cH)$, and a subspace $\sM\subset \cH$. Then
\[
\sM\cap \ran X^{1/2}=\sM^\perp\cap\ran X^{1/2}=\{0\}
\]
if and only if the operator $X$ with respect to decomposition $\cH=\sM\oplus\sM^\perp$ takes the form
\begin{equation}
\label{crab1} X=\begin{bmatrix} W^2&WU\cr U^*W&U^*U  \end{bmatrix},
\end{equation}
where
\begin{equation}
\label{crab2}
\begin{array}{l} W\in\bB_{0}^+(\sM),
\;U\in\bB(\sM^\perp,\sM),\; \ker U=\{0\}, \;\ker U^*=\{0\},\\
\ran W\cap\ran U=\{0\}.
\end{array}
\end{equation}
\end{corollary}
\begin{proof}
If $X$ is of the form \eqref{crab1} with conditions \eqref{crab2}
then due to \eqref{quadrform}, \eqref{nol}, \eqref{posshor1} we get
that
\[
\sM\cap \ran X^{1/2}=\sM^\perp\cap\ran X^{1/2}=\{0\},\; \ker
X=\{0\}.
\]
Conversely, suppose $X\in\bB_{0}^+(\cH)$ and $\ran
X^{1/2}\cap\sM=\ran X^{1/2}\cap\sM^\perp=\{0\}.$ From \eqref{nol}
and \eqref{shmul1},  \eqref{shmul2} we get that with respect to orthogonal
decomposition $\cH=\sM\oplus\sM^\perp$ the operator $X$ takes the form
\[
X=\begin{bmatrix}X_{11}&X^{1/2}_{11}\Gamma X^{1/2}_{22}\cr
X^{1/2}_{22}\Gamma^* X^{1/2}_{11}&X_{22}
\end{bmatrix},
\]
where $\ker X_{11}=\{0\}$, $\ker X_{22}=\{0\}$, and $\Gamma$ is
unitary map of $\sM^\perp$ onto $\sM$. Denote $W=X_{11}^{1/2}$,
$U=\Gamma X^{1/2}_{22}$. Then $X$ is of the form \eqref{crab1}.
Moreover, $\ran W\cap\ran U=\{0\}$ due to $\ker X=\{0\}$. Therefore, conditions \eqref{crab2} are satisfied.
\end{proof}
Let $U=V\Phi$ be the polar decomposition of $U\in\bB(\sM^\perp,\sM)$, where $V=(UU^*)^{1/2}$ and $\Phi$ is
unitary operator acting from $\sM^\perp$ onto $\sM$. Then $X$ in \eqref{crab1} takes the form
\begin{equation}
\label{crab44}
 X=\begin{bmatrix} W^2&WV\Phi\cr \Phi^*VW&\Phi^*V^2\Phi
\end{bmatrix}:\begin{array}{l}\sM\\\oplus\\ \sM^\perp\end{array}\to
\begin{array}{l}\sM\\\oplus\\ \sM^\perp\end{array},
\end{equation}
and
\begin{equation} \label{crab3}
\begin{array}{l} W\in\bB_{0}^+(\sM),
\;V\in\bB_{0}^+(\sM), \\
\ran W\cap\ran V=\{0\}.
\end{array}
\end{equation}

Let us formulate a \textit{general} criterion: The operator $X\in\bB_{0}^+(\cH)$, having the block-operator matrix form
\[
X=\begin{bmatrix} X_{11}&X_{12}\cr X_{12}^*&X_{22}
\end{bmatrix}:\begin{array}{l}\sM\\\oplus\\ \sM^\perp\end{array}\to
\begin{array}{l}\sM\\\oplus\\ \sM^\perp\end{array},
\]
possess the property
\[
\sM\cap\ran X^{1/2}=\sM^\perp\cap\ran X^{1/2}=\{0\} \ ,
\]
if and only if
\[
\ker X_{11}=\{0\}, \; \ker X_{22}=\{0\},\; \ran X_{12}\cap\ran
X_{11}=\{0\}.
\]

{Now we are in position to formulate the first of our main results of this section. It concerns subspaces that have
trivial intersections with the operator range $\ran A^{1/2}$.}
\begin{theorem}
\label{SchmTh333} Let $A\in\bB_{0}^+(\cH)$ and $\ran A\ne \cH$.
{Then there is a continuum set of subspaces $\sM\subset \cH$ such
that}
\begin{equation}
\label{zeroint111} \sM\cap \ran A^{1/2}=\sM^\perp\cap\ran A^{1/2}=\{0\}.
\end{equation}
\end{theorem}
\begin{proof} By Theorem \ref{NeumTh2} there exists
$B\in\bB_{0}^+(\cH)$ such that $\ran B^{1/2}\cap\ran A^{1/2}=\{0\}$ (take for instance $B=UAU^*$, where $U$ is
unitary in $\cH$ and satisfies \eqref{fw11}). Then the parallel sum $A:B=0$, and hence by Theorem \ref{shm1} there is
a subspace $\sM$ such that
$$\sM\cap\ran (A+B)^{1/2}=\sM^\perp\cap\ran (A+B)^{1/2}=\{0\}.$$
Since $\ran (A+B)^{1/2}=\ran A^{1/2}+\ran B^{1/2},$ we get
\eqref{zeroint111}. Notice that
\[
A=(A+B)^{1/2}P_\sM(A+B)^{1/2},
\;B=(A+B)^{1/2}P_{\sM^\perp}(A+B)^{1/2}.
\]
Let $M\in\bB^+(\cH)$. Then
$$\ran(B^{1/2}MB^{1/2})^{1/2}=B^{1/2}\ran
M^{1/2}\subseteq\ran B^{1/2}.$$
Since $\ran B^{1/2}\cap \ran
A^{1/2}=\{0\}$, then
$$\ran (B^{1/2}MB^{1/2})^{1/2}\cap\ran A^{1/2}=\{0\} \ .$$
Hence
\[
\begin{array}{l}
A=(A+B^{1/2}MB^{1/2})^{1/2}P(M)(A+B^{1/2}MB^{1/2})^{1/2},\\
B^{1/2}MB^{1/2}=(A+B^{1/2}MB^{1/2})^{1/2}(I-P(M))(A+B^{1/2}MB^{1/2})^{1/2},
\end{array}
\]
where $P(M)$ is orthogonal projection in $\cH$. In particular
\begin{equation}
\label{potx} A=(A+xB)^{1/2}P(x)(A+xB)^{1/2}, \;xB=(A+xB)^{1/2}(I-P(x))(A+xB)^{1/2} \ ,
\end{equation}
for proportional to identity operators $M(x) = x I$ with positive
parameter $x$. Here we put $P(x):=P(M(x))$. Then
\[
\ran P(x)\cap\ran A^{1/2}=\ran(I-P(x))\cap\ran A^{1/2}=\{0\}.
\]
We show first that if $x,y>0$ and
$x\ne y$, then $P(x)\ne P(y)$. Notice that
\begin{equation}
\label{dosnt}
\ran (A+x B)^{1/2}=\ran A^{1/2}\dot+\ran B^{1/2}, \; x>0.
\end{equation}
Suppose that $x<y$. Then
\[
A+x B\le A+y B.
\]
Hence, by the Douglas Theorem \ref{Douglas} one obtains
\begin{equation}
\label{operzed} (A+x B)^{1/2}=Z_{x,y}(A+y B)^{1/2},
\end{equation}
where $Z_{x,y}\in\bB(\cH)$ is a contraction. Note that
\[
Z^*_{x,y}= (A+y B)^{-1/2}(A+x B)^{1/2}.
\]
Then equality \eqref{dosnt} implies that the operators $Z^*_{x,y}$
as well as $Z_{x,y}$ are isomorphisms of $\cH$. The first equality
in \eqref{potx} yields that
\[
A=(A+yB)^{1/2}Z^*_{x,y}P(x)Z_{x,y}(A+yB)^{1/2}.
\]
On the other hand
\[
A=(A+yB)^{1/2}P(y)(A+yB)^{1/2}.
\]
Thus
\begin{equation}
\label{ravpro}
P(y)=Z^*_{x,y}P(x)Z_{x,y}.
\end{equation}
Set $\sM_x=\ran P(x)$, $\sM_y=\ran P(y).$ From \eqref{ravpro} we get
that
\begin{enumerate}
\item $Z_{x,y}$ maps $\sM^\perp_y$ into $\sM^\perp_x$,
\item $Z_{x,y}$ maps $\sM_y$ into $\sM_x$ isometrically.
\end{enumerate}
In fact $Z_{x,y}\sM_y=\sM_x$ and $Z_{x,y}\sM^\perp_y=\sM^\perp_x$,
because $Z_{x,y}$ is isomorphism of $\cH$. The equalities
\[
\begin{array}{l}
xB=(A+y B)^{1/2}Z^*_{x,y}(I-P(x))Z_{x,y}(A+y B)^{1/2},\\
yB=(A+y B)^{1/2}(I-P(y))(A+y B)^{1/2}
\end{array}
\]
lead to
\[
y\,Z^*_{x,y}(I-P(x))Z_{x,y}=x\,(I-P(y)),
\]
and taking into account \eqref{ravpro} we arrive at
\[
(y-x)P(y)= y Z^*_{x,y}Z_{x,y}- xI.
\]
Finally
\[
y||Z_{x,y} h||^2={x}||h||^2,\; h\in\sM^\perp_y.
\]
This equality means that the operator
\[
\sqrt{y x^{-1}}\,Z_{x,y}
\]
isometrically maps $\sM^\perp_y$ onto $\sM^\perp_x$. Now assume
$P(x)=P(y)$, i.e., $\sM_x=\sM_y.$ Denote this subspace by $\sM.$
Then with respect to decomposition $\cH=\sM\oplus\sM^\perp$ the operator $Z_{x,y}$ takes the matrix form
\[
Z_{x,y}=\begin{bmatrix}\Lambda_1&0\cr 0&\sqrt{xy^{-1}}\,\Lambda_2
 \end{bmatrix}:\begin{array}{l}\sM\\
 \oplus\\
 \sM^\perp\end{array}\to \begin{array}{l}\sM\\
 \oplus\\
 \sM^\perp\end{array},
\]
where $\Lambda_1$ and $\Lambda_2$ are unitary operators in $\sM$ and
$\sM^\perp$, respectively.
 Since
\[
A=(A+yB)^{1/2}P_\sM(A+yB)^{1/2}=(A+xB)^{1/2}P_\sM(A+yB)^{1/2},
\]
from \eqref{operzed} follows
\[
A=Z_{x,y}(A+yB)^{1/2}P_\sM(A+yB)^{1/2}Z^*_{x,y}.
\]
Therefore
\[
A=Z^*_{x,y}AZ_{x,y}.
\]
Due to the structure of $Z_{x,y}$, the latter equality implies
\[
P_\sM^\perp A\uphar\sM^\perp =Z^*_{x,y}(P_\sM^\perp
A\uphar\sM^\perp)Z_{x,y}\uphar\sM^\perp.
\]
Hence
\[
\|P_\sM^\perp A\uphar\sM^\perp\|=xy^{-1}\|\Lambda^{-1}_2(
P_\sM^\perp A\uphar\sM^\perp)\Lambda_2\|.
\]
Because $\Lambda_2$ is unitary in $\sM^\perp$, we get
\[
\|\Lambda^{-1}_2( P_\sM^\perp
A\uphar\sM^\perp)\Lambda_2\|=\|P_\sM^\perp A\uphar\sM^\perp\|.
\]
Thus,
\[
\|P_\sM^\perp A\uphar\sM^\perp\|=xy^{-1}\|P_\sM^\perp
A\uphar\sM^\perp\|.
\]
Since $x\ne y$, this equality implies: $P_\sM^\perp A \, \uphar \, \sM^\perp=0$, i.e. the contradiction
with $A\in\bB^+_0(\cH)$. So, $P(x)\ne P(y)$ if $x\ne y$.
\end{proof}
{It should be noted that the function $P(x)$ is \textit{strongly} continuous at each point on $(0,+\infty)$. To prove this
we define the following auxiliary operator-valued function
\[
S_x: =(A+xB)^{-1/2}A^{1/2}\ , \  \  x > 0 \ .
\]
Then
\[
S^*_x h=A^{1/2}(A+xB)^{-1/2}h, \ \  h\in\ran A^{1/2}\dot+\ran B^{1/2} \ ,
\]
and from \eqref{potx} one gets
\begin{equation}
\label{shift} S_xS^*_x=P(x).
\end{equation}
Let $x_0>0$, then
\begin{multline*}
\left(S^*_x-S^*_{x_0}\right)(A+x_0B)f =(x_0-x)A^{1/2}(A+x
B)^{-1/2}Bf\\
+A^{1/2}\left((A+x B)^{1/2}-(A+x_0B)^{1/2}\right)f.
\end{multline*}
Note that the Douglas Theorem \ref{Douglas} implies: $\sqrt{x}||(A+xB)^{-1/2}B^{1/2}||\le 1$. Therefore,
\[
||(A+x B)^{-1/2}B^{1/2}||\le C
\]
for all $x$ in some neighborhood of the point $x_0$. Hence,
\[
\lim\limits_{x\to x_0}\left(S^*_x-S^*_{x_0}\right)(A+x_0B)f=0
\]
for all $f\in\cH$. Since the linear manifold $\ran (A+x_0B)$ is
dense in $\cH$ and $||S^*_x||=1$, we obtain
\[
\lim\limits_{x\to x_0}S^*_xg=S^*_{x_0}g
\]
for all $g\in\cH$. From \eqref{shift} it follows that
$||S_x^*g||=||P(x)g||$, then
\begin{equation}
\label{norma} \lim\limits_{x\to x_0}||P(x)g||=||P(x_0)g||,\;
g\in\cH.
\end{equation}
On the other hand,
\[
(P(x)g-P(x_0)g,f)=(S_xS^*_x g,f)-(S_{x_0}S^*_{x_0}g,f) =(S^*_xg,
S^*_x f)-(S^*_{x_0}g, S^*_{x_0} f).
\]
Hence, the function $P(x)$ is weakly continuous at $x_0$, which together with \eqref{norma} implies that $P(x)$
is strongly continuous at $x_0$. }
\begin{proposition}\label{11}
Let $\{U_t\}_{t\in \mathbb{R}}$ be a one-parameter unitary group such that
$$U_t \, \ran A^{1/2} \cap U_s \, \ran A^{1/2}=\{0\},\; s\ne t,$$
see Corollary \ref{coro1}, then there is a one-parameter family of
subspaces $\{\sM_t\}_{t\in \mathbb{R}\setminus\{0\}}$ such that
\[
\sM_t\cap\ran A^{1/2}=\sM^\perp_t\cap\ran A^{1/2}=\{0\} \ \ \mbox{for all}\; t\ne 0 \ .
\]
Moreover,
\begin{equation}
\label{pmt} P_{\sM^\perp_{-t}}=U_{-t}P_{\sM_t}U_t, \; t\ne 0,
\end{equation}
and therefore
\[
\sM^\perp_{-t}=U_{-t}\sM_t,\; t\ne 0.
\]
\end{proposition}
\begin{proof}
Consider $B_t=U_tAU_{-t}$. Note that by the Stone theorem $U_t$ is
of the form $U_t=\exp(it H)$, $t\in\dR$, where generator $H$ is a
self-adjoint operator in $\cH$. Since
$$\ran B^{1/2}_t\cap\ran
A^{1/2}=\{0\},\; t\ne 0,$$
we get
\[
\begin{array}{l}
A=(A+B_t)^{1/2}P_t (A+B_t)^{1/2},\\
B_t=(A+B_t)^{1/2}(I-P_t)(A+B_t)^{1/2},\; t\ne 0,
\end{array}
\]
where $P_t := P_{\sM_t}$ are orthogonal projections on $\sM_t \subset \cH$ for all $t\ne 0$.
Then
\begin{equation}
\label{pt} A= U_{-t}(A+B_t)^{1/2}(I-P_t)(A+B_t)^{1/2}U_{t}.
\end{equation}
By virtue of equality
\[
U_{-t}(A+B_t)=(B_{-t}+A)U_{-t}
\]
we get
\[
U_{-t}(A+B_t)^k=(B_{-t}+A)^kU_{-t}
\]
for all $k\in\dN$. Hence
\[
U_{-t}(A+B_t)^{1/2}=(B_{-t}+A)^{1/2}U_{-t},\; t\ne 0.
\]
Then \eqref{pt} yields
\[
A=(B_{-t}+A)^{1/2}U_{-t}(I-P_t)U_t(B_{-t}+A)^{1/2}, \; t\ne 0.
\]
Since also
\[
A=(A+B_{-t})^{1/2}P_{-t} (A+B_{-t})^{1/2},\; t\ne 0,
\]
we obtain  \eqref{pmt} with $\sM_t:=\ran P_t$.
\end{proof}
{Next we show that there exists \textit{increasing} (\textit{decreasing}) chains of subspaces possessing the trivial
intersection property \eqref{zeroint111}.}
\begin{theorem}
\label{monot} Let $A\in\bB_{0}^+(\cH)$ and $\ran A\ne \cH$.
Then there is an increasing sequence $\sN_1\subset\sN_2\subset\ldots$ of subspaces in $\cH$ such that
\begin{enumerate}
\item $\sN_k\cap\ran A^{1/2}=\sN^\perp_k\cap \ran A^{1/2}=\{0\}$ for all
$k\in\dN,$
\item $\bigcap\limits_{k\in\dN}\sN^\perp_k=\{0\},$
\item $s-\lim\limits_{k\to\infty} P_{\sN_k}=I_\cH.$
\end{enumerate}
\end{theorem}
\begin{proof}
(1) Choose $\sM\subset\cH$ such that $\sM\cap\ran A^{1/2}=\sM^\perp\cap\ran A^{1/2}=\{0\}$ and define
\[
A_1:=A^{1/2}P_\sM A^{1/2},\; A_2:=A^{1/2}_1P_\sM A^{1/2}_1,\;\ldots A_k:=A^{1/2}_{k-1}P_\sM A^{1/2}_k,\; \ldots.
\]
Then
\begin{equation}\label{A>A}
A\ge A_1\ge A_2\ge\cdots \ ,
\end{equation}
and $\ran A^{1/2}_k=A^{1/2}_{k-1}\sM\subset \ran A^{1/2}_{k-1}$ for $k\in\dN$. It follows that
$$
\begin{array}{l}
\ran A^{1/2}\supset \ran A^{1/2}_1\supset\ran
A^{1/2}_2\supset\cdots,\\
\sM\cap \ran A^{1/2}_k=\sM^\perp\cap \ran A^{1/2}_k=\{0\}, \; k\in\dN \ .
\end{array}
$$
In addition, by Proposition \ref{shm01}(1), $\ran A^{1/2}_1\cap\ran A=\{0\}$. Hence,
\begin{equation}
\label{polzn}
 \ran A^{1/2}_k\cap\ran A=\{0\},\; k\in\dN.
\end{equation}
Now from Proposition \ref{mars1} and (\ref{A>A}) it follows by induction that
\[
A_k=A^{1/2}P_kA^{1/2}, \; k\in\dN \ ,
\]
where $\{P_k\}_{k\in\dN}$ are orthogonal projections in $\cH$ with $P_1:=P_\sM \ $ such that
$$P_1\ge P_2\ge \ldots  \ge P_k\ge \ldots  \ . $$
Since $\ker A_k=\{0\}$, one gets $\ran (I-P_k)\cap\ran
A^{1/2}=\{0\}$. Equation \eqref{polzn} and Proposition \ref{shm01}
yield also that $\ran P_k \cap \ran A^{1/2}=\{0\}.$ Set
\[
\sN_1=\sM^{\perp}, \; \sN_k=\ran(I-P_k),\; k\in\dN.
\]
Then we obtain
\[
\sM^\perp=\sN_1\subset\sN_2\subset\ldots,\; \sM=\sN^\perp_1\supset\sN^\perp_2\supset\ldots \ .
\]
(2) Now let us show that
\[
\bigcap\limits_{k\in\dN}\sN^\perp_k=\{0\}.
\]
{Note first that the sequence $\{A_k\}_{k\geq 1}\subset\bB^+(\cH)$ is non-increasing. So, it has the strong limit
$A_0:= s-\lim_{k\to\infty}A_k$  and $\ran A^{1/2}_0\subset\ran A^{1/2}$.} Therefore, from
$A_k=A^{1/2}_{k-1}P_\sM A^{1/2}_k$, $k\in\dN$ we obtain $A_0=A^{1/2}_0P_\sM A^{1/2}_0$. Hence $\ran A^{1/2}_0\subset\sM.$
On the other hand $\ran A^{1/2}_0\cap\sM=\{0\}$, which implies that operator $A_0=0$.

Suppose $f\in \bigcap\limits_{k\in\dN}\sN^\perp_k$, i.e.,
$f=A^{1/2}_k f_k$. The equality $A_k=A^{1/2}_{k-1}P_\sM A^{1/2}_k$
and Proposition \ref{shm01} imply that
\[
||f_k||=||A^{-1/2}f|| \ \ \ \mbox{for all} \ \ k\in\dN.
\]
Since $s-\lim_{k\to\infty}A_k=0$ and
\[
(f, h)=(f_k, A^{1/2}_k h),\; k\in\dN,
\]
we get $f=0$. Thus $\bigcap\limits_{k\in\dN}\sN^\perp_k=\{0\}$.\\
(3) Moreover, since $s-\lim_{k\to\infty}P_k=0$, we also get $s-\lim_{k\to\infty}P_{\sN_k}=I_\cH.$
\end{proof}
Note that Theorem \ref{SchmTh333} can be reformulated in terms of the operator ranges as follows.
\begin{theorem}
\label{SchmTh3} Let operator range $\cR$ be non-closed and dense in
a Hilbert space $\cH$. Then there is a subspace $\sM\subset \cH$
such that
\[
 \sM\cap\cR=\sM^\perp\cap\cR=\{0\}.
\]
\end{theorem}
\begin{proof}
Let $A\in\bB_{0}^+(\cH)$, and $\ran A^{1/2}=\cR$. Then apply Theorem \ref{SchmTh333}.
\end{proof}
Concequently, by applying Proposition \ref{shm01} and Theorem \ref{SchmTh333} to $A=B^2$, where $B\in\bB^+(\cH)$, one
can now prove the Van Daele Theorem \ref{daele82-2}.
\begin{corollary}
Let $\{F_j\}_{j=1}^n\in\bB^+(\cH)$, $\ker \left(\sum\limits_{j=1}^n F_j\right)=\{0\}$, and $\ran \left(\sum\limits_{j=1}^n
F_j\right)\ne\cH$. Then there are infinitely many subspaces $\sM$
such that
\[
\sM\cap\ran\left(\sum\limits_{j=1}^n F_j\right)^{1/2}=
\sM^\perp\cap\ran\left(\sum\limits_{j=1}^n F_j\right)^{1/2}=\{0\}.
\]
In particular
\[
\sM\cap\ran F^{1/2}_j=\sM^\perp\cap\ran F^{1/2}_j=\{0\} \;\mbox{for
all}\; j=1,2\ldots, n.
\]
\end{corollary}
\begin{corollary}
\label{nov1} For arbitrary operator $A\in\bB^{+}_{0}(\cH)$ with  $\ran A\ne \cH$ there exists infinitely
many pairs $A_1,$ $A_2\in\bB^+(\cH)$  such that
\begin{equation} \label{aaa}
\begin{array}{l} A=A_1+A_2,\\
\ker A_1=\ker A_2=\{0\},\\
\ran A^{1/2}_1\cap\ran A^{1/2}_2=0,\\
\ran A^{1/2}_1\dot+\ran A^{1/2}_2=\ran A^{1/2}.
\end{array}
\end{equation}
\end{corollary}
\begin{proof}
Let $\sM$ be a subspace and $\sM\cap\ran A^{1/2}=\sM^\perp\cap\ran
A^{1/2}=\{0\}$. Define two operators
\[
A_1=A^{1/2}P_\sM A^{1/2},\; A_2= A^{1/2}P_{\sM^\perp}A^{1/2}.
\]
Then equalities in \eqref{aaa} are satisfied.
\end{proof}

{Since by definition for any operator $A\in\bB^{+}(\cH)$  the set of all \textit{extreme points} of the
operator interval $[0,A]$ are of the form}
\[
\{A^{1/2}PA^{1/2}:\;{P} \; \mbox{is an arbitrary orthogonal
projection in}\;\cH\},
\]
see \cite{Pek}, the statement of Corollary \ref{nov1} has the following interpretation:

\textit{There exists infinitely many pairs $\left<X, A-X\right>$ of extreme points of the operator interval $[0, A]$ such
that $\ker X=\ker(A-X)=\{0\}.$ Moreover, there are increasing {\rm{(}}decreasing{\rm{)}} sequences $\{X_n\}_{n\geq1}$ of
such extreme points, which in addition have the property
$s-\lim_{n\to\infty} X_n=A$ {\rm{(}}$s-\lim _{n\to\infty} X_n=0${\rm{)}}.}

\subsection{Lifting of operators}\label{LiftOper}
Let $A\in\bB_{0}^+(\cH)$ with $\ran A\ne \cH$. For a given subspace $\sM$, possessing the property
\eqref{zeroint111}, and for the corresponding orthogonal projection $P_\sM$, we are looking for existence of representation
of the operator $A$ in the form:
\begin{equation}
\label{repr11}
\begin{array}{l}
 A=T^{1/2}P_\sM T^{1/2},\\
 T\in\bB_{0}^+(\cH),\\
 \; \sM\cap\ran
T^{1/2}=\sM^\perp\cap\ran T^{1/2}=\{0\}.
\end{array}
\end{equation}
We call this representation the \textit{lifting of operator} $A$ and we refer to the operator $T$ as to the
\textit{lifting operator} for a given subspace $\sM$.

{The following statement, which makes our concept of lifting nontrivial} can be easily derived from
Proposition \ref{shm1} and Theorem \ref{SchmTh3}.
\begin{proposition}\label{lift1}
Let $A\in\bB_{0}^+(\cH)$ with $\ran A\ne \cH$. Then operator $A$ admits a lifting in the form
$$
A=T^{1/2}PT^{1/2} ,
$$
where $T\in\bB_{0}^+(\cH)$ and $P$ is an orthogonal projection in $\cH$ such that
\[
\ran P\cap\ran T^{1/2} = \ran(I-P)\cap\ran T^{1/2}=\{0\} \ .
\]
\end{proposition}
\noindent Notice that from the Proposition \ref{lift1} one also obtains the triviality of intersections:
\[
\ran P\cap\ran A^{1/2}=\{0\} \ \  \ {\rm{and}} \ \  \ \ran(I-P)\cap\ran A^{1/2}=\{0\}.
\]

For the following we need an auxiliary statement concerning the operator ranges, which we formulate as the
lemma.
\begin{lemma} \label{obr1} Let operator range $\cR$ be dense in $\cH$. Then there are operators
$Y\in \bB_{0}^+(\cH)$ such that
\[
\ran Y\cap\cR=\{0\}\ \ \ {\rm{and}} \ \ \ \ran Y^{1/2}\supset\cR \ .
\]
\end{lemma}
\begin{proof} Let $Z\in \bB^+(\cH)$ with $\cR=\ran Z^{1/2}$. Since $\cR$ is dense, we have $\ker Z=\{0\}$.
{Then by Proposition \ref{root} one can find $X\in\bB^+(\cH)$, such that}
$$\ker X=\{0\} \ \ \ {\rm{and}} \ \ \ \ran X^{1/2}\cap\ran Z^{1/2}=\{0\} ,$$
and $Z=(Z+X)^{1/2}P(Z+X)^{1/2}$, where $P$ is an orthogonal projection. Set $Y:=Z+X$.
By construction we
have $\ran Y^{1/2}\supset\ran Z^{1/2}.$  Since $\ker X =\ker Z=\{0\}$, by Proposition \ref{shm01} we get $\ran Y\cap\ran
Z^{1/2}=\{0\}$ and the proof is completed.
\end{proof}
\begin{theorem}
\label{obr2} Let $A\in\bB_{0}^+(\cH)$ and let $\sM$ be a subspace
possessing \eqref{zeroint111}. Suppose that the block operator-matrix $A$ is of the form:
\[
A=\begin{bmatrix} A_{11}&A_{12}\cr A_{12}^*&A_{22}
\end{bmatrix}:\begin{array}{l}\sM\\\oplus\\ \sM^\perp\end{array}\to
\begin{array}{l}\sM\\\oplus\\ \sM^\perp\end{array}.
\]
Then operator $A$ admits the lifting in the form \eqref{repr11} for the subspace $\sM$ if and only if
\begin{equation}
\label{raninc} \ran A_{12}\subset \ran A^{3/4}_{11}.
\end{equation}
\end{theorem}
\begin{proof}
By Corollary \ref{VAZZN} the block operator-matrix $A$ with respect to decomposition $\cH=\sM\oplus\sM^\perp$ is of the
form \eqref{crab1}:
\[ A=\begin{bmatrix} W^2&WU\cr U^*W&U^*U
\end{bmatrix}:\begin{array}{l}\sM\\\oplus\\ \sM^\perp\end{array}\to
\begin{array}{l}\sM\\\oplus\\ \sM^\perp\end{array},
\]
where the matrix entry are
\[
A_{11}=W^2,\; A_{12}=WU,\; A^*_{12}=U^*W,\; A_{22}=U^*U.
\]
Note that by \eqref{crab2}, $\ran W\cap\ran U=\{0\}$, i.e., $\ran
A^{1/2}_{11}\cap \ran(A^{-1/2}_{11}A_{12})=\{0\}$.

Suppose that representation \eqref{repr11} is valid for some  $T\in\bB_{0}^+(\cH)$.
By virtue of \eqref{shmul1}  the operator $T^{1/2}$ has a matrix form with respect to decomposition
$\cH=\sM\oplus\sM^\perp$:
\[
T^{1/2}=\begin{bmatrix} X_{11}&X^{1/2}_{11}G X^{1/2}_{22}\cr
X^{1/2}_{22}G^* X^{1/2}_{11}&X_{22}\end{bmatrix},
\]
where $G\in\bB(\sM^\perp,\sM)$ is a contraction. Since $\ker T=\{0\}$, one has $\ker X_{11}=\{0\}$ and $\ker X_{22}=\{0\}$.
Hence
\[
\begin{bmatrix} W^2&WU\cr U^*W&U^*U
\end{bmatrix}=A=T^{1/2}P_\sM T^{1/2}=\begin{bmatrix} X^2_{11}& X_{11}X_{12}\cr
X^*_{12}X_{11}& X^*_{12}X_{12}\end{bmatrix},
\]
where $X_{12}:=X^{1/2}_{11}G X^{1/2}_{22}$  and consequently
\[
X_{11}=W, \ X_{12} = U = W^{1/2}G X^{1/2}_{22}, \ \ran W^{1/2}\supset\ran U \ .
\]
Therefore, the inclusion \eqref{raninc} holds.

Now assume \eqref{raninc} and define $M:=W^{-1/2}U=A^{-3/4}_{11}A_{12}$. Since $\ran U\cap\ran W=\{0\}$,
we get $\ran M\ne \cH$. The latter and the equality $\ker U=\{0\}$ imply that $\ran M^*$ is dense in $\cH$ but
$\ran M^*\ne \cH$. Now, by Lemma \ref{obr1} one can find $Y\in \bB_{0}^+(\sM^\perp)$  such that
$$
\ran Y\cap \ran M^*=\{0\}\quad\mbox{and}\quad\ran Y^{1/2}\supset\ran
M^*.
$$
Define $Q:=tY^{-1/2}M^*$, where $t>0$ is such that $||Q||\le 1$ and set
$$
X_{22}:=Y/t^2,\; G:=Q^*.
$$
Then $M^*=X^{1/2}_{22}G^*$, $M=GX^{1/2}_{22}$ and finally $U=W^{1/2}GX^{1/2}_{22} $.

Let us introduce
\[
L:=\begin{bmatrix}W&W^{1/2}G X^{1/2}_{22}\cr
X^{1/2}_{22}G^*W^{1/2}&X_{22}
\end{bmatrix}=\begin{bmatrix}W&U\cr
U^*&X_{22}\end{bmatrix}.
\]
Then $L\in \bB_{0}^+(\cH)$ and the equality
\[
A=LP_\sM L=\begin{bmatrix}W^2&WU\cr U^*W&U^*U\end{bmatrix}
\]
holds. If now we define $T:=L^2$, then $A=T^{1/2}P_\sM T^{1/2}$, where
\begin{multline*}
T^{1/2}\begin{bmatrix} f_1\cr f_2\end{bmatrix}=\begin{bmatrix}
Wf_1+Uf_2\cr U^*f_1+X_{22}f_2\end{bmatrix}\\
=\begin{bmatrix}Wf_1+Uf_2\cr
X^{1/2}_{22}G^*W^{1/2}f_1+X_{22}f_2\end{bmatrix}
=\begin{bmatrix}Wf_1+Uf_2\cr M^*W^{1/2}f_1+X_{22}f_2\end{bmatrix}
\end{multline*}
Since $\ran U\cap\ran W=\{0\}$ and $\ran X_{22}\cap\ran M^*=\{0\}$, we obtain
$$
\sM\cap\ran T^{1/2}=\sM^\perp\cap\ran T^{1/2}=\{0\},
$$
which yields representation \eqref{repr11}.
\end{proof}

The next statement follows from Corollary \ref{VAZZN}, Theorem \ref{obr2}, and \eqref{crab44}, \eqref{crab3}.
\begin{corollary}
\label{cledstv} Let $A\in\bB_{0}^+(\cH)$. Let $\sM$ be a subspace in $\cH$ such that
$\dim\sM=\dim{\sM^\perp}=\infty$. Then operator $A$ admits the lifting in the form \eqref{repr11} for the subspace
$\sM$ if and only if
\[
 A=\begin{bmatrix} W^2&WV\Phi\cr \Phi^*VW&\Phi^*V^2\Phi
\end{bmatrix}:\begin{array}{l}\sM\\\oplus\\ \sM^\perp\end{array}\to
\begin{array}{l}\sM\\\oplus\\ \sM^\perp\end{array},
\]
where
\begin{equation}
\label{otnos1}
\begin{array}{l} W\in\bB_{0}^+(\sM),
\;V\in\bB_{0}^+(\sM), \\
 \Phi\quad\mbox{unitarily maps}\quad\sM^\perp\quad\mbox{onto}\quad\sM,\\
\ran V\cap\ran W=\{0\},\\
\ran V\subset\ran W^{1/2}.
\end{array}
\end{equation}
Similarly, the operator $A$ admits the lifting in the form
\begin{equation}
\label{repr22}
\begin{array}{l}
 A=Q^{1/2}P_{\sM^\perp} Q^{1/2},\\
 Q\in\bB_{0}^+(\cH),\\
 \; \sM\cap\ran
Q^{1/2}=\sM^\perp\cap\ran Q^{1/2}=\{0\}
\end{array}
\end{equation}
if and only if
\begin{equation}
\label{otnos11}
\begin{array}{l}
\ran V\cap\ran W=\{0\},\\
\ran V^{1/2}\supset\ran W.
\end{array}
\end{equation}
Finally, the operator $A$ admits the lifting in the both forms \eqref{repr11} and \eqref{repr22} if and only if
\begin{equation}
\label{otnos2}
\begin{array}{l}
\ran V\cap\ran W=\{0\},\\
\ran W^{1/2}\supset\ran V,\\
 \ran V^{1/2}\supset\ran W.
\end{array}
\end{equation}
\end{corollary}

One can resume the above observations as following:\\  Let $W\in\bB_{0}^+(\sM)$ with $\ran W\ne \sM.$
Then there exists a subspace
$\sL\subset\sM$ such that
\begin{equation}
\label{l-inter} \sL \cap \ran W^{1/2}=\sL^\perp \cap \ran
W^{1/2}=\{0\}.
\end{equation}
(a) Define the operator
\[
V_1:=W^{1/2}P_\sL W^{1/2}.
\]
Then one obtains that
$$\ran V^{1/2}_1\cap\ran W=\{0\} \ \   {\rm{and}}  \ \ \ran V_1\subset\ran W^{1/2},$$
i.e., the operator $V_1$ satisfies \eqref{otnos1}, but it does \textit{not} satisfy \eqref{otnos11}.
This means that for any unitary mapping $\Phi$ of $\sM^\perp$ onto $\sM$ the operator
\[
A_1:=\begin{bmatrix} W^2&WV_1\Phi\cr \Phi^*V_1W&\Phi^*V^2_1\Phi
\end{bmatrix}:\begin{array}{l}\sM\\\oplus\\ \sM^\perp\end{array}\to
\begin{array}{l}\sM\\\oplus\\ \sM^\perp\end{array}
\]
admits the lifting \eqref{repr11} by $T$, but it does not admit the lifting by $Q$ in the form \eqref{repr22}.\\
(b) {Let us define
\[
 V_2:=W^{1/2}(I+P_\sL)W^{1/2}.
\]
Using \eqref{usef} and the equality $\ran (I+P_\sL)=\sM$, we get
that $\ran V^{1/2}_2=\ran W^{1/2}.$ On the other hand if $V_2x=Wy$,
then $(I+P_\sL)W^{1/2}x=W^{1/2}y$. It follows that $P_\sL
W^{1/2}x=W^{1/2}(y-x)$. Condition \eqref{l-inter} yields that
$y=x=0$. Hence, $\ran V_2\cap\ran W=\{0\}$, i.e., the operator $V_2$
satisfies \eqref{otnos2}. Consequently, the operator
$$A_2:=\begin{bmatrix} W^2&WV_2\Phi\cr \Phi^*V_2W&\Phi^*V^2_2\Phi
\end{bmatrix}:\begin{array}{l}\sM\\\oplus\\ \sM^\perp\end{array}\to
\begin{array}{l}\sM\\\oplus\\ \sM^\perp\end{array},$$
admits the lifting in the form \eqref{repr11} and in the form \eqref{repr22} for any unitary $\Phi$
}.\\
(c) {Choose $V\in\bB^+_0(\sM)$ such that $\ran V^{1/2}\cap\ran W^{1/2}=\{0\}$. Then operator $V$ satisfies the condition
$\ran V\cap\ran W=\{0\}$, see \eqref{crab3}, but it does \textit{not} satisfies both conditions \eqref{otnos1} and
\eqref{otnos11}. Therefore, operator $A$ does \textit{not} admits the lifting in the form \eqref{repr11}
\textit{and} in the form \eqref{repr22}.} This example indicates a \textit{limit} for application of our method.

\subsection{Applications to unbounded operators}\label{Appl}

First we present here an extended version and a new proof of the Schm\"{u}dgen  Theorem \ref{SchmTh1}. The both
follow from our results in Section \ref{triv-intersec} and \ref{LiftOper}.

\begin{theorem} \label{111} {Let $H$ be a closed unbounded densely defined linear operator in a Hilbert
space $\cH$. Then
\begin{enumerate}
\item
there exists a subspace $\sM$ of $\cH$ such that
\begin{equation}
\label{novdok} \sM\cap\dom H=\sM^\perp\cap\dom H=\{0\} \ ,
\end{equation}
moreover, there exists uncountably many of them ,
\item there exists a fundamental symmetry $J$ in $\cH$ such that
\begin{equation}
\label{novfund}
J \, \dom H \cap \dom H= \{0\},
\end{equation}
moreover, there exists uncountably many of them.
\end{enumerate}
}
\end{theorem}
\begin{proof}
Let $A=(H^*H+I)^{-1}$. Then $A\in\bB^+(\cH)$ and $\cR:=\ran A^{1/2}=\dom H.$ By Theorem \ref{SchmTh333},
Proposition \ref{11}, and Theorem \ref{SchmTh3} there exists uncountable set of subspaces
$\sM $ of $\cH$ satisfying \eqref{novdok}. {Therefore, combining this observation with Proposition \ref{dobavl} we
deduce that there exists uncountable set of fundamental symmetries $J$ satisfying \eqref{novfund}. }
\end{proof}

Note that by virtue of Theorem \ref{SchmTh3} and Proposition \ref{11} there exists a one-parameter family
$\{\sM_t\}_{t\in \mathbb{R}}$ of subspaces {and operators $\{J_t:= (2 P_{\sM_t} - I)\}_{t\in \mathbb{R}}$ satisfying
respectively \eqref{novdok} and \eqref{novfund}.}
{Then besides Proposition \ref{dobavl} we can formulate the following version of the von Neumann Theorem \ref{NeumTh1}.}
\begin{corollary}
\label{Neum Th11} {For any unbounded self-adjoint operator $H$ in a Hilbert space $\cH$ there exists
a one-parameter family $\{U_t\}_{t\in \mathbb{R}}$ of unitary operators with the property
\[
\dom H\cap\dom (U_t ^*H U_t)=\{0\} \ .
\]
Moreover, one can find a strongly continuous family $\{J_t\}_{t\in \mathbb{R}}$ of fundamental symmetries
(i.e. self-adjoint and unitary operators) such that
\[
\dom H\cap\dom (J_t H J_t)=\{0\} \ .
\]
}
\end{corollary}
\begin{corollary}
\label{1111}  Let $T_1, \ldots, T_n$ be closed unbounded densely
defined linear operators in a Hilbert space $\cH$. Then there exists
infinitely many subspaces $\sM$ in $\cH$ such that
\[
 \sM\cap\left(\sum\limits_{j=1}^n\dom T_k\right)
 =\sM^\perp\cap\left(\sum\limits_{j=1}^n\dom T_k\right)=\{0\}.
\]
In particular
\[
\sM\cap\dom T_j=\sM^\perp\cap\dom T_j=\{0\}\;\mbox{for all}\;
j=1,\ldots, n.
\]
\end{corollary}

Let $T$ be non-negative self-adjoint operator in $\cH$. As it is well-known \cite{Ka}, \cite{Kr1} the sesquilinear form
$(Tu,v),$ $u,v\in\dom T$ admits a closure and  we (following Kre\u\i n \cite{Kr1}) denote this closure by $T[\cdot,\cdot]$
and its domain by $\cD[T]$ (see Subsection \ref{FrKR}). Then by the \textit{second representation theorem} \cite{Ka}
\[
\cD[T]=\dom T^{1/2} \ \ {\rm{and}} \ \ T[f,g]=(T^{1/2}f,T^{1/2}g),\; f, g \in \dom T^{1/2}.
\]
The linear manifold $\cD[T](=\dom T^{1/2})$ is the Hilbert space
with respect to the graph inner product
\begin{equation} \label{inpr}
(f,g)_{T^{1/2}}=T[f,g]+(f,g).
\end{equation}

The fractional-linear transformation
\[
S=(I-T)(I+T)^{-1},\; T=(I-S)(I+S)^{-1}
\]
gives a one-to-one correspondence between the set of all non-negative self-adjoint operators $T$ and
the set of all self-adjoint contractions $S$ such that $\ker (S+I)=\{0\}$, see \cite{Kr1}. Then one can easily derive
\cite{AHS2} that
\begin{equation}\label{CLOSFORM}
\begin{array}{l}
\cD[T ]=\ran(I+S)^{1/2},\\
T [u,v]=-(u,v)+2\left((I+S)^{-1/2}u,(I+S)^{-1/2}v\right), \quad
u,v\in \cD[T ].
\end{array}
\end{equation}

The next application of our approach is the following theorem and the corresponding remarks.
\begin{theorem} \label{appl1} Let $T$ be unbounded non-negative self-adjoint operator
in Hilbert space $\cH$. Then there are infinitely many pairs $\left
<T_1, T_2\right>$ of unbounded non-negative self-adjoint operators
such that
\begin{enumerate}
\item $\dom T^{1/2}_1\cap\dom T=\dom T^{1/2}_2\cap\dom T=\{0\};$
\item $\dom T^{1/2}_k\subset\dom T^{1/2},$  and $||T^{1/2}_k g||=||T^{1/2} g||$
for all $g\in\dom T^{1/2}_k$, $k=1,2$;
\item $\dom T^{1/2}_1\cap\dom T^{1/2}_2=\{0\}$;
\item the Hilbert space $\cD[T]$ admits the orthogonal decomposition $\cD[T]=\cD[T_1]\oplus_{T^{1/2}}\cD[T_2]$
with respect to the inner product (\ref{inpr}).
\end{enumerate}
\end{theorem}
\begin{proof}
Let $S=(I-T)(I+T)^{-1}$ and let $\sM$ be a subspace in $\cH$ such that  (see Corollary \ref{111})
\[
\sM\cap\dom T^{1/2}=\sM^\perp\cap\dom T^{1/2}=\{0\}.
\]
We define
\begin{equation}
\label{defin}
S_1=(I+S)^{1/2}P_\sM(I+S)^{1/2}-I\ , \ \ S_2=(I+S)^{1/2}P_{{\sM}^\perp}(I+S)^{1/2}-I.
\end{equation}
The operators $S_1$ and $S_2$ are self-adjoint contractions with $\ker (S_k +I)=\{0\}$, $k=1,2.$ Let
\[
T_k=(I-S_k)(I+S_k)^{-1},\; k=1,2.
\]
Then $T_1$ and $T_2$ are non-negative self-adjoint operators.
Using \eqref{CLOSFORM} and \eqref{defin} we have
\begin{gather*}
\dom T^{1/2}=\ran (I+S)^{1/2},\\
\dom T^{1/2}_1=(I+S)^{1/2}\sM,\; \dom
T^{1/2}_2=(I+S)^{1/2}\sM^\perp.
\end{gather*}
Notice that by definitions
\[
\begin{array}{l}
\dom T^{1/2}_1\cap\dom T^{1/2}_2=\{0\},\\
\dom T^{1/2}_1\dot+\dom T^{1/2}_2=\dom T^{1/2}.
\end{array}
\]
Suppose $(I+S_1)^{1/2}u=(I+S)f$, i.e., $(I+S)^{1/2}x=(I+S)f$
for some $x\in \sM$. Hence $x=(I+S)^{1/2}f$. But $\ran
\sM\cap(I+S)^{1/2}=\{0\}$. This means that $u=f=0$ and, therefore, $\dom
T^{1/2}_1\cap\dom T=\{0\}$. Similarly  $\dom T^{1/2}_2\cap\dom T=\{0\}.$

From $I+S_1=(I+S)^{1/2}P_\sM(I+S)^{1/2}$
 we obtain
\[
(I+S_1)^{1/2}h=(I+S)^{1/2}\cU h,\; h\in H,
\]
where $\cU$ is unitary operator from $\cH$ onto $\sM(=\ran P_\sM)$.
Hence
\[
(I+S)^{-1/2}g=\cU (I+S_1)^{-1/2}g\quad\mbox{for all}\quad
g\in\ran(I+S_1)^{1/2}=\dom T^{1/2}_1.
\]
Thus,
\begin{equation}
\label{vspom11}
||(I+S_1)^{-1/2}g||^2=||(I+S)^{-1/2}g||^2,\;g\in\ran(I+S_1)^{1/2}.
\end{equation}
Now \eqref{CLOSFORM} and \eqref{vspom11} yield
$||T^{1/2}_1g||=||T^{1/2} g||$ for all $g\in\dom T^{1/2}_1.$
Similarly $||T^{1/2}_2g||=||T^{1/2} g||$ for all $g\in\dom
T^{1/2}_2.$

Let $u\in\dom T^{1/2}_1$, $v\in\dom T^{1/2}_2$. Then
\[
\begin{array}{l}
u=(I+S)^{1/2}f, \;f\in\sM,\\
v=(I+S)^{1/2}h, \;h\in\sM^\perp.
\end{array}
\]
From  \eqref{inpr} and \eqref{CLOSFORM} we get $(u,v)_{T^{1/2}}=0.$
This yields the orthogonal decomposition
$\cD[T]=\cD[T_1]\oplus_{T^{1/2}}\cD[T_2]$ and the proof is
completed.
\end{proof}
\begin{remark} \label{svezh}
From the proof one can also find the expressions of $T_1$ and $T_2$ via $T$:
\[
\begin{array}{l}
T_1=\left((I+T)^{-1/2}P_\sM(I+T)^{-1/2}\right)^{-1}-I,\\
T_2=\left((I+T)^{-1/2}P_{\sM^\perp}(I+T)^{-1/2}\right)^{-1}-I.
\end{array}
\]
Then
\[
(I+T_1)^{-1}+(I+T_2)^{-1}=(I+T)^{-1}.
\]
This means that any vector $f_T\in\dom T$ admits a unique
decomposition
\[
f_T=f_{T_1}+f_{T_2},
\]
where $f_{T_1}\in\dom T_1$ and $f_{T_2}\in\dom T_2$, although
\[
\dom T\cap\dom T_1=\dom T\cap\dom T_2=\{0\}.
\]
Here $f_{T_1}=(I+T_1)^{-1}(I+T)f_T$, $f_{T_2}=(I+T_2)^{-1}(I+T)f_T$.
 In addition, the
following equalities are valid:
\[
\begin{array}{l}
\dom T^{1/2}_1=(I+T)^{-1/2}\sM,\\
\dom T^{1/2}_2=(I+T)^{-1/2}\sM^\perp.
\end{array}
\]
\end{remark}
\begin{remark} \label{zamech}
{\rm{(a)}} {Theorem \ref{appl1} yields also that $\ker T_1=\ker
T_2=\{0\}$. Suppose that $f\neq 0$ and $T_1f=0$, then (2) implies $Tf=0$, i.e.,
$\ker T_1\subseteq\ker T$. This gives $\dom T^{1/2}_1\cap\dom T\ne\{0\}$, which contradicts to (1).}\\
{\rm{(b)}} The equalities \eqref{doug11} and \eqref{defin} imply that
\[ \ran T^{1/2}_{1}=\ran T^{1/2} + \dom T^{1/2}_2,\;\ran T^{1/2}_{2}=\ran T^{1/2} + \dom T^{1/2}_1.
\]
In particular
\[
\ran T^{1/2}_1\supseteq\ran T^{1/2},\;\ran T^{1/2}_2\supseteq\ran
T^{1/2}.
\]
\end{remark}
\begin{remark}\label{zamech-bis}
{\rm{(a)}} By the properties (2) and (3) the form-sum $T_k \dot + T$
is equal to operator $2T_k$, for $k=1,2$. Then the Lie-Trotter-Kato
product formula \cite{Ka1}, \cite{Ka2} implies the strong
convergence
\[
s-\lim\limits_{n\to\infty}\left(\exp({-tT/n})\exp({-tT_k/n})\right)^n=\exp({-2t
T_k}),\;t\ge 0,\; k=1,2.
\]
{\rm{(b)}} Since $\dom T^{1/2}_1\cap\dom T^{1/2}_2=\{0\}$, the Kato theorem \cite[Theorem 1]{Ka1} yields
\[
s-\lim\limits_{n\to\infty}\left(\exp({-tT_1/n})\exp({-tT_2/n})\right)^n=
s-\lim\limits_{n\to\infty}\left(\exp({-tT_2/n})\exp({-tT_1/n})\right)^n=0.
\]
\end{remark}
{Note also that Theorem \ref{appl1} and Remark \ref{svezh} yield the following statement.}
{\begin{corollary} \label{nov111}
Let $T$ be unbounded non-negative self-adjoint operator in Hilbert space $\cH$. Then for each natural number $n$ there
exists $n$ unbounded non-negative self-adjoint operators $\{T_k\}_{k=1}^{n}$ such that
\begin{enumerate}
\item $\dom T^{1/2}_k\cap\dom T=\{0\},$ $k=1,2,\ldots,n,$
\item if $k\ne j$, then $\dom T^{1/2}_k\cap\dom T^{1/2}_j=\{0\}$,
\item the form $T_k[\cdot,\cdot]$ is a closed restrictions of the
form $T[\cdot,\cdot],$
\item
$\cD[T]=\cD[T_1]\oplus_{T^{1/2}}\cD[T_2]\oplus_{T^{1/2}}\cdots\oplus_{T^{1/2}}\cD[T_n],$
\item $(T+I)^{-1}=(T_1+I)^{-1}+(T_2+I)^{-1}+\cdots+(T_n+I)^{-1}.$
\end{enumerate}
\end{corollary}
}
\begin{theorem} \label{obratn}
Let {$T_1$} be unbounded non-negative self-adjoint operator in $\cH$ with $\ker T_1=\{0\}$. Then one can always find two
non-negative self-adjoint operators {$T_2$ and $T$} such that conditions (1)--(4) of Theorem \ref{appl1} are satisfied.
\end{theorem}
\begin{proof} Let
$$S_1:=(I-T_1)(I+T_1)^{-1},\; A_1=\cfrac{1}{2}\ (I+S_1).$$
Then $0\le A_1\le I$, $\ker A_1=\ker (I-A_1)=\{0\}$. By Theorem \ref{NeumTh2} there exists $X\in\bB_{0}^+(\cH)$ such
that $\ran X^{1/2}\cap\ran A^{1/2}_1=\{0\}$, $0\le X\le I$. Set
\[
B:=(I-A_1)^{1/2}X(I-A_1)^{1/2}.
\]
Then $\ker B=\{0\}$, $0\le A_1+ B\le I$, $\ker(A_1+ B)=\{0\}$,  and
the equalities
\begin{multline*}
 \ran B^{1/2}=(I-A_1)^{1/2}\ran X^{1/2},\;\ran (A_1-
A^{2}_1)^{1/2}=\ran A^{1/2}_1\cap\ran (I-A_1)^{1/2},\\
\ran X^{1/2}\cap\ran A^{1/2}_1=\{0\}
\end{multline*}
imply $\ran B^{1/2}\cap\ran A^{1/2}_1=\{0\}$.  Hence
\[
A_1=(A_1+B)^{1/2}P(A_1+B)^{1/2},\;B=(A_1+B)^{1/2}(I-P)(A_1+B)^{1/2} \ ,
\]
where $P$ is some  orthogonal projection in $\cH,$ see Proposition \ref{root}. Then define
\[
\begin{array}{l}
S:=2(A_1+B)-I, \; S_2:=2B-I,\\
T:=(I-S)(I+S)^{-1},\; T_2:=(I-S_2)(I+S_2)^{-1}.
\end{array}
\]
Since $I+S_1=(I+S)^{1/2}P(I+S)^{1/2}$ and
$I+S_2=(I+S)^{1/2}(I-P)(I+S)^{1/2}$, one follows arguments used in Theorem \ref{appl1} to complete the proof.
\end{proof}
The next statement is extension of Theorem \ref{appl1} to the  \textit{infinite} family of operator pairs.
\begin{theorem} \label{now11} Let $T$ be unbounded non-negative self-adjoint operator
in Hilbert space $\cH$. Then there are pairs of families
$\left<\{T_{1,j}\}_{j\in\dN},\{T_{2,k}\}_{k\in\dN}\right>$ of
unbounded non-negative self-adjoint operators possessing the
following properties:
\begin{enumerate}
\item $\dom T^{1/2}\supset\dom T^{1/2}_{1,1}\supset\dom T^{1/2}_{1,2}\supset\cdots\supset\dom T^{1/2}_{1,j}
\supset\cdots$,
\item $\bigcap\limits_{j\in\dN}\dom T^{1/2}_{1,j}=\{0\},$
\item
$\dom T^{1/2}_{1,j}\cap\dom T=\{0\}$ 
$( T_{1,0}=T_{2,0}=T),$
\item $\dom T^{1/2}_{2,1}\subset\dom T^{1/2}_{2,2}\subset\cdots\subset\dom T^{1/2}_{2,j}\subset\cdots
\subset\dom T^{1/2}$,
\item $\dom T^{1/2}_{2,j}\cap\dom T=\{0\}$ for all $j\in\dN$,
\item the sesquilinear forms $ T_{1,j}[\cdot,\cdot]$ and  $T_{2,j}[\cdot,\cdot]$ are closed restrictions of the form
$T[\cdot,\cdot]$ for each $j\in\dN$,
\item $\cD[T]=\cD[T_{1,j}]\oplus_{T^{1/2}}\cD[T_{2,j}]$ for each $j\in\dN$,
\item $s-\lim\limits_{j\to\infty}( T_{1,j}-\lambda I)^{-1}=0$, for $\lambda\in\dC\setminus\dR_+$,
\item $s-\lim\limits_{j\to\infty}(T_{2,j}-\lambda I)^{-1}=(T-\lambda I)^{-1}$, for $\lambda\in\dC\setminus\dR_+,$
\item $s-\lim\limits_{j\to\infty}\exp(-z T_{1,j})=0$ for all $z$, $\RE
z>0$, 
\item $s-\lim\limits_{j\to\infty}\exp(-z T_{2,j})=\exp(-zT)$ for all $z$, $\RE
z {\ge} 0$.
 \end{enumerate}
\end{theorem}
\begin{proof} Let $S:=(I-T)(I+T)^{-1}$. Then by Theorem \ref{monot} there
is an increasing sequence $\sN_1\subset\sN_2\subset\ldots$ of
subspaces in $\cH$ such that
\begin{enumerate}
\item $\sN_k\cap\ran (I+S)^{1/2}=\sN^\perp_k\cap \ran (I+S)^{1/2}=\{0\}$ for all
$k\in\dN,$
\item $\bigcap\limits_{k\in\dN}\sN^\perp_k=\{0\},$
\item $s-\lim\limits_{k\to\infty} P_{\sN_k}=I_\cH.$
\end{enumerate}
Then we define
\[
S_{1,j}:=(I+S)^{1/2}P_{\sN_j^\perp}(I+S)^{1/2}-I \ \ {\rm{and}} \ \ S_{2,j}:=(I+S)^{1/2}P_{\sN_j}(I+S)^{1/2}-I,\; j\in\dN.
\]
Due to $\sN^\perp_{1}\supset\sN^\perp_2\supset\ldots$ and $\sN_1\subset\sN_2\subset\ldots$ one obtains
$$I+S\ge I+S_{1,1}\ge
I+S_{1,2}\ge\ldots,\;I+S_{2,1}\le I+S_{2,2}\le\ldots\le I+S.
$$
Hence for $j\in\dN$:
\[
\begin{array}{l}
\ran (I+S_{1,j+1})^{1/2}\subset\ran(I+S_{1,j})^{1/2}\subset\ran
(I+S)^{1/2},\\
\ran (I+S_{2,j})^{1/2}\subset\ran(I+S_{2,j+1})^{1/2}\subset\ran
(I+S)^{1/2} .
\end{array}
\]
Consequently, we get for  $j,l,k\in\dN$:
\[
\begin{array}{l}
\ran(I+S_{1,j})^{1/2}\cap\sN_k=\ran(I+S_{1,j})^{1/2}\cap\sN^\perp_k=\{0\},\\
\ran(I+S_{2,l})^{1/2}\cap\sN_k=\ran(I+S_{2,l})^{1/2}\cap\sN^\perp_k=\{0\}.
\end{array}
\]
Since $s-\lim\limits_{j\to\infty}P_{\sN_j}=I$, one also has
\[
s-\lim\limits_{j\to\infty}(I+ S_{1,j})=0 \ \ {\rm{and}} \ \ s-\lim\limits_{j\to\infty}(I+ S_{2,j})=I+S \ .
\]
Now we define for $j\in\dN$:
$$T_{1,j}:=(I-S_{1,j})(I+
S_{1,j})^{-1} \ \ {\rm{and}} \ \ T_{2,j}:=(I-S_{2,j})(I+ S_{2,j})^{-1} .
$$
Then $\{T_{1,j}\}$ and $\{T_{2,j}\}$ are non-negative self-adjoint
operators, such that (see Theorem \ref{appl1} and Theorem
\ref{monot})
\[ \dom T^{1/2}_{1,j}\cap\dom T=\{0\},\; \dom
T^{1/2}_{2,j}\cap\dom T=\{0\},\; j\in\dN,
\]
and properties (1)--(7) hold true. Since
$(T_{1,j}+I)^{-1}=2(I+S_{1,j})$ and $(T_{2,j}+I)^{-1}=2(I+S_{2,j})$,
we obtain
$$
s-\lim\limits_{j\to\infty}(T_{1,j}+I)^{-1}=0\ \ {\rm{and}} \ \ s-\lim\limits_{j\to\infty}(T_{2,j}+I)^{-1}=2(I+S)=(I+T)^{-1} .
$$
This implies (see \cite[Chapter VIII, Theorem 1.3]{Ka}) that
\[
\begin{array}{l}
s-\lim\limits_{j\to\infty}( T_{1,j}-\lambda I)^{-1}=0,\\
s-\lim\limits_{j\to\infty}( T_{2,j}-\lambda I)^{-1}=(T-\lambda
I)^{-1}
\end{array}
\]
for $\lambda\in\dC\setminus\dR_+$.

In order to prove for all $z$, $\RE z>0$, the limit $s-\lim\limits_{j\to\infty}\exp(-z T_{1,j})=0$ we use the Euler
approximation of the one-parameter semigroup  $\{\exp(-tA)\}_{t\ge 0}$ with $m-\alpha$ sectorial ($\alpha\in[0,\pi/2)$)
generator $A$ in the operator-norm topology \cite{ArlZag2010},\cite{Zag2008}:
\[
||\exp(-t A)-(I+tA/n)^{-n}||\le \cfrac{K_\alpha}{n\,\cos^2\alpha},\;
t\ge 0,\; n\in\dN,
\]
Here $K_\alpha$ depends only on $\alpha$, see \cite{CD}.

Let $\wt T$ be non-negative self-adjoint operator and $\RE z\ge0$. Then for $z=t\,e^{i\f}$ and $|\f|\in[0,\pi/2)$
the operator $A=e^{i\f}\wt T$ is $m-|\f|$-sectorial generator. Put $\wt T=T_{1,j}$. Then
\begin{multline*}
||\exp(-z T_{1,j})f||\le||\left(\exp(-t(e^{i\f}
T_{1,j}))-(I+t(e^{i\f} T_{1,j})/n)^{-n}\right)f||\\
+||(I+t(e^{i\f}T_{1,j})/n)^{-n}f||\\
\le \cfrac{K_{|\f|}}{n\,\cos^2\f}\,||f||+||(I+t(e^{i\f}
T_{1,j})/n)^{-n}f||\\
\le \cfrac{K_{|\f|}}{n\,\cos^2\f}\,||f|| + C(n,\f,t)||(I+\cfrac{te^{i\f}}{n}T_{1,j})^{-1}f|| ,
\end{multline*}
for any $f\in \mathcal{H}$ and some constant $C(n,\f,t)$.
Since above it was established that for each $n$, $\f$, $f$, and $t>0$
\[
\lim\limits_{j\to \infty}||(I+\cfrac{te^{i\f}}{n} T_{1,j})^{-1}f||=0 \ ,
\]
we obtain $s-\lim\limits_{j\to\infty}\exp(-z  T_{1,j})=0$.

Applying now the Trotter-- Kato approximation theorem \cite[Chapter III, Section 4.9]{EN}, we obtain
\[
s-\lim\limits_{j\to\infty}\exp(-z T_{2,j})=\exp(-zT)\;\mbox{for all}
\;z,\; \RE z\ge 0 \ ,
\]
and the end of the proof.
\end{proof}
\begin{remark}
{\rm{1)}} The inclusions $\dom  T^{1/2}_{1,j-1}\supset\dom
T^{1/2}_{1,j}$ and equalities $|| T^{1/2}_{1,j}f||= || T^{1/2}f||$
for all $f\in\dom T^{1/2}_{1,j}$ and all $j\in\dN$ mean that
\[
 T\le T_{1,1}\le T_{1,2}\le\ldots\le  T_{1,k}\le\ldots\ ,
\]
in the sense of associated closed quadratic forms \cite{Ka}. It is
proved by T.~Kato \cite[Lemma 1]{Ka1}, that if $H_j$ for $j\in\dN$
are self-adjoint operators in $\cH$ such that $I_\cH\le H_1\le
H_2\le\ldots $ and
$$\cD_0=\{u\in\cH: u\in\bigcap_{j\in\dN}\dom H^{1/2}_j, \;\sup_j||H^{1/2}_j u||<\infty\},
$$
then $\lim\limits_{j\to\infty}H^{-1}_jv=0$ for all $v\perp\cD_0.$

{\rm{2)}} Let ${\bf H}:=\left\{\left<0,h\right>,\;h\in\cH\right\}$. Then ${\bf H}$ is a self-adjoint linear relation
\cite{Ar}. The resolvent $({\bf H}-\lambda I)^{-1}$: $\cH\to\cH$ is identically zero operator. The equality
$s-\lim\limits_{j\to\infty}( T_{1,j}-\lambda I)^{-1}=0$ means that in the strong resolvent limit sense \cite{Ka}
the sequence of operators $\{T_{1,j}\}_{j\in\dN}$ converges to ${\bf H}$.

{\rm{3)}} From Remark \ref{zamech} it follows that for all
$j\in\dN$:
\[
\ran T^{1/2}_{1,j}=\ran T^{1/2}+\dom T^{1/2}_{2,j}\ \ {\rm{and}} \ \
\ran T^{1/2}_{2,j}=\ran T^{1/2}+\dom T^{1/2}_{1,j}.
\]
In particular,
\[
\begin{array}{l}
\ran T^{1/2}\subseteq\ran T^{1/2}_{1,1}\subseteq \ran
T^{1/2}_{1,2}\subseteq\ldots,\; T^{-1}\ge T^{-1}_{1,1}\ge T^{-1}_{1,2}\ge\ldots,\\
\ran T^{1/2}_{2,1}\supseteq \ran T^{1/2}_{2,2}\supseteq
\ldots\supseteq\ran T^{1/2},\;  T^{-1}_{2,1}\le T^{-1}_{2,2}\le\ldots \leq T^{-1}, \\
\end{array}
\]
and
\begin{enumerate}
\item $s-\lim\limits_{j\to\infty}(T^{-1}_{1,j}-\lambda I)^{-1}=-\lambda^{-1} I$,
\item $s-\lim\limits_{j\to \infty} e^{-zT^{-1}_{1,j}}=I,$ for all $z$, $\RE
z>0$,
\item $s-\lim\limits_{j\to\infty}(T^{-1}_{2,j}-\lambda I)^{-1}=(T^{-1}-\lambda I)^{-1}$,
for $\lambda\in\dC\setminus\dR_+,$ {\rm{(}}note that $T^{-1}$ is, in general, a linear relation{\rm{)}}\ ,
\item $\lim\limits_{j\to \infty}
T^{-1}_{2,j}[u,v]=T^{-1}[u, v]$ for all $u,v\in\ran T^{1/2}$.
\end{enumerate}
\end{remark}
\subsection{Beyond the Van Daele--Schm\"{u}dgen and the Brasche--Neidhardt theorems}\label{VD-Sch-B-N}
We start by theorem, which is a weaker (but useful) version of the Van Daele Theorem \ref{daele82-2}.
\begin{theorem}
\label{nov113} Let $B$ be a non-negative unbounded self-adjoint operator in a Hilbert space $\cH$. Suppose that
$\ker B=\{0\}$ and $\ran B\ne \cH.$ Then there exists two linear manifolds $\sD_1$ and $\sD_2$ possessing the
following properties:
\begin{enumerate}
\item $\sD_1\dot+\sD_2=\dom B,$
\item $B \, \sD_1$ and $B \, \sD_2$ are dense in $\cH.$
\end{enumerate}
Moreover, one can choose $\sD_1$ and $\sD_2$ such that
$\sD_1\oplus_B\sD_2=\dom B.$
\end{theorem}
\begin{proof} Set $S:=(I-B^2)(I+B^2)^{-1}$.
Then
$$B^2=(I-S)(I+S)^{-1},\; B=(I-S)^{1/2}(I+S)^{-1/2}.
$$
Hence
\[
\dom B=\ran (I+S)^{1/2},\; B(I+S)^{1/2}h=(I-S)^{1/2}h,\; h\in\cH.
\]
Since $\ker (I-S)=\{0\}$ and $\ran (I-S)\ne \cH$, there exists a
subspace $\sM$ of $\cH$ such that
\begin{equation}
\label{intflt} \sM\cap\ran
(I-S)^{1/2}=\sM^\perp\cap\ran(I-S)^{1/2}=\{0\}.
\end{equation}
Let
\[
\sD_1:=(I+S)^{1/2}\sM \ \ \ {\rm{and}} \ \ \ \sD_2:=(I+S)^{1/2}\sM^\perp.
\]
It follows then from \eqref{inpr} and \eqref{CLOSFORM} that
\[
\sD_1\oplus_B\sD_2=\dom B.
\]
Since $B \, \sD_1=(I-S)^{1/2} \sM,$ $B \, \sD_2=(I-S)^{1/2}\sM^\perp,$ and \eqref{intflt} holds, we conclude that the sets
$B \, \sD_1$ and $B \, \sD_2$ are dense in $\cH$.
\end{proof}

 In the following theorem our approach elucidates the property of products and squares of unbounded
operators.
\begin{theorem}
\label{prodzer} Let $T$ and $\left<T_1, T_2\right>$ be as in Theorem \ref{appl1}. Define non-negative self-adjoint
operator $\cL:=T^{1/2}$ together with two of its densely defined symmetric restrictions
\[
\dot{\cL}_1:=T^{1/2}\uphar\dom
T^{1/2}_1 \ \ {\rm{and}} \ \ \dot{\cL}_2:=T^{1/2}\uphar\dom T^{1/2}_2.
\]
Then
\begin{enumerate}
\item operators $\dot{\cL}_1$ and $\dot{\cL}_2$ are
closed,
\item $\dom\dot{\cL_1}\cap\dom \cL^2=\dom\dot{\cL_2}\cap\dom
\cL^2=\{0\}$,
\item $\dom\dot{\cL}_1\cap\dom\dot{\cL}_2=\{0\}$ and
  $\dom \cL=\dom\dot{\cL}_1\dot+\dom\dot{\cL}_2$,
\item
$ \dom(\cL\dot{\cL}_1)=\dom(\cL\dot{\cL}_2)=\{0\}, $ in particular,
$ \dom\dot{\cL}^2_1=\dom \dot{\cL}^2_2=\{0\}.
$
\end{enumerate}
If $\ran T=\cH$, then
\begin{enumerate}
\item
$\dot{\cL}_k\cL$ is densely defined,
\item $(\dot{\cL}_k\cL)^*=\cL\dot{\cL}^*_k$, $k=1,2$,
\item $\dom(\dot{\cL}_1\cL)\cap\dom(\dot{\cL}_2\cL)=\{0\}$ and
\[
\dom(\dot{\cL}_1\cL) \dot+ \dom(\dot{\cL}_2\cL)=\dom \cL^2,
\]
\item the operator $\cL^2(=T)$ is the Friedrichs extension of $\dot{\cL}_1\cL$ and $\dot{\cL}_2\cL$,
\item the Kre\u\i n extension of the operator $\dot{\cL}_j\cL$ is
$(\dot{\cL}_j\cL)_{\rm K}= \dot{\cL}_j\dot{\cL}^*_j$, $ j=1,2.$
\end{enumerate}
\end{theorem}
\begin{proof} Since for all $\f\in\dom T^{1/2}_1=\dom \dot{\cL}_1$ one has
\[
||\dot{\cL}_1\f||^2=||T^{1/2}\f||^2=||T^{1/2}_1\f||^2 ,
\]
we get that $\dot{\cL}_1$ is closed operator and the \textit{first representation theorem} \cite{Ka} leads to
equality $\dot{\cL}^*_1\dot{\cL}_1=T_1$. Similarly the operator $\dot{\cL}_2$ is closed and $\dot{\cL}^*_2\dot{\cL}_2=T_2.$
Taking into account that $\dom T_k\subset\dom T^{1/2}_k$ and $\dom T^{1/2}_k\cap\dom T=\{0\}$ we obtain
$\dom T_k\cap\dom T=\{0\}$,
$k=1,2.$ This means that
 \[
\dom\dot{\cL_1}\cap\dom \cL^2=\dom\dot{\cL_2}\cap\dom \cL^2=\{0\}.
 \]
 Hence
 \begin{equation}
 \label{predpr}
\dom(\dot{\cL}^*_k\dot{\cL}_k)\cap\dom \cL^2=\{0\}, k=1,2.
 \end{equation}
The condition $\cL\supset\dot{\cL}_k$ leads to equality
$$
\dom(\dot{\cL}^*_k\dot{\cL}_k)\cap \dom \cL^2=\dom (\cL\dot{\cL}_k).
$$
Then \eqref{predpr} yields $\dom(\cL\dot{\cL}_k)=\{0\}$, $k=1,2$.

Suppose that $\ran T=\cH$. Then the operators $\cL$, $\dot{\cL}_1\cL$,
and $\dot{\cL}_2\cL$ are \textit{positive definite}, see Section \ref{Int}. It follows then that
\[
\dom(\dot{\cL}_k\cL)=\cL^{-1}\dom\dot{\cL}_k \ \ {\rm{and}} \ \
(\dot{\cL}_k\cL)(\cL^{-1}\f)=\dot{\cL}_k\f,\;\f\in\dom\dot{\cL}_k,\;
k=1,2.
\]
This yields, that $\dom(\dot{\cL}_k\cL)$ is dense in $\dom\cL$ with respect to the graph-norm in $\dom \cL$.
Hence, the operator $\dot{\cL}_k\cL$ is densely defined in $\cH$ and, moreover, the
Friedrichs extensions of $\dot{\cL}_k\cL$ (for k=1,2) coincide with operator $\cL^2$, see Section \ref{FrKR}
and Theorem \ref{Sim}.

Note that $\ker(\dot{\cL}_k\cL)^* =\ker \dot{\cL}^*_k$. Therefore, relations
\[
\dom\dot{\cL}_k^*=\dom\cL\dot+\ker \dot{\cL}^*_k \ \ {\rm{and}} \ \
\dom(\dot{\cL}_k\cL)^*=\dom\cL^2\dot+\ker (\dot{\cL}_k\cL)^*
\]
lead to the equality $(\dot{\cL}_k\cL)^*=\cL\dot{\cL}^*_k$, $k=1,2$.

The equality $\dom\dot{\cL}_1\dot+\dom\dot{\cL}_2=\dom\cL$ implies
that $\dom(\dot{\cL}_1\cL)\cap\dom({\dot\cL}_2\cL)=\{0\}$ and
\[
\dom(\dot{\cL}_1\cL)\dot+\dom(\dot{\cL}_2\cL)\subseteq\dom \cL^2.
\]
Let $f\in\dom\cL^2=\dom T$. Then $\cL f=T^{1/2}f\in\dom T^{1/2}$.
Due to the direct decomposition
\[
\dom T^{1/2}_1\dot+\dom T^{1/2}_2=\dom T^{1/2},
\]
we get $T^{1/2}f=\f_1+\f_2$, where $\f_k\in\dom T^{1/2}_k$, $k=1,2.$ The equality $\ran T^{1/2}=\cH$ implies that
$\f_k=T^{1/2}h_k$, $h_k\in\dom T^{1/2}$. Since $\f_k=T^{1/2}h_k\in\dom T^{1/2}_k$, $k=1,2$, we obtain that
$h_k\in\dom(\dot{\cL}_k\cL)$, $k=1,2$, and therefore
\[
f=T^{-1/2}\f_1+T^{-1/2}\f_2=h_1+ h_2.
\]
So we proved
\[
\dom \cL^2\subseteq \dom(\dot{\cL}_1\cL)\dot+\dom(\dot{\cL}_2\cL),
\]
which implies
\[
\dom \cL^2= \dom(\dot{\cL}_1\cL)\dot+\dom(\dot{\cL}_2\cL).
\]

Finally, since $\ran \cL=\cH$, one gets $(\dot{\cL}_j\cL)_{\rm
K}=\dot{\cL}_j\dot{\cL}^*_j$ for $j=1,2$ (see Theorem \ref{Sim}).
\end{proof}
\begin{remark}
Abstract examples of pairs $\langle\cL_0, \cL\rangle$: $\cL_0\subset\cL$, consisting of a densely defined closed and
non-negative symmetric operator $\cL_0$ and of its non-negative self-adjoint extension $\cL$ such that
\begin{enumerate}
\item $\dom(\cL\cL_0)=\{0\},$
\item $\cL_0\cL$ is densely defined,
$(\cL_0\cL)^*=\cL\cL^*_0,$ and the operator $\cL^2$ is the Friedrichs extension of $\cL_0\cL$
\end{enumerate}
are given in \cite{ArlKov_2013}.
\end{remark}
The next two assertions are strengthened versions of the Van
Daele--Schm\"{u}dgen and Brasche--Neidhardt theorems \cite{Daele2},
\cite{schmud}, \cite{BraNeidh} mentioned in Section \ref{Int}.
\begin{theorem}
\label{novsch} Let $B$ be unbounded self-adjoint operator in a
Hilbert space $\cH$. Then there are infinitely many pairs
$\left<B_1, B_2\right>$ of  densely defined closed restrictions of
$B$ such that
\begin{enumerate}
\item $\dom B_1\cap\dom B_2=\{0\}$;
\item $\dom B=\dom B_1\dot+\dom B_2$;
\item $\dom B_1\cap\dom B^2=\dom B_2\cap\dom B^2=\{0\};$
\item $\dom (BB_1)=\dom (BB_2)=\{0\}$, and in particular,  $\dom
B^2_1=\dom B^2_2=\{0\}$.
\end{enumerate}
If $\ran B=\cH$, then
\begin{enumerate}
\item $B_kB$ is densely defined, k=1,2,
\item the operator $B^2 = (\dot{B}_1 B)_{\rm F} = (\dot{B}_2 B)_{\rm F}$ is the Friedrichs extension of the
operators $\dot{B}_1 B$ and $\dot{B}_2 B$,
\item $(B_kB)^*=BB^*_k,$ $k=1,2$,
\item $\dom (B_1B)\cap\dom (B_2 B)=\{0\}$ and
\[
\dom(B_1B) \dot+ \dom(B_2B)=\dom B^2,
\]
\item the Kre\u\i n extension of the operator $\dot{B}_j B$ is  $(\dot{B}_j B)_{\rm K}=\dot{B}_j \dot{B}^*_j$, $ j=1,2.$
\end{enumerate}
\end{theorem}
\begin{proof}
Let $T=B^2$. Then there exists a pair $\left<T_1, T_2\right>$ possessing  properties 1)--4) mentioned in Theorem \ref{appl1}.
Now let
\[
B_1:=B\uphar\dom T^{1/2}_1 \ \ {\rm{and}} \ \  B_2:=B\uphar\dom T^{1/2}_2.
\]
Then $\dom B_1\cap\dom B_2=\{0\}$. In addition $\dom B_1\cap\dom
B^2=\dom B_2\cap\dom B^2=\{0\}.$ Because $\dom B=\dom
\sqrt{B^2}=\dom T^{1/2}$ and $\dom T^{1/2}=\dom T^{1/2}_1\dot+\dom
T^{1/2}_2$ we get
 $\dom B=\dom B_1\dot+\dom B_2$.

 Arguing as in Theorem \ref{prodzer} and taking into account
 that $||B\f||^2=||\sqrt{B^2}\f||,$ $\f\in\dom B$, we get
$B^*_kB_k=T_k,$ $k=1,2.$ Since $\dom T_1\cap\dom T=\dom T_2\cap \dom
T=\{0\}$, we get
\[
\dom (B^*_1B_1)\cap\dom B^2=\dom (B^*_2B_2)\cap\dom B^2=\{0\},
\]
Therefore, since $B^*_1\supset B$, $B^*_2\supset B$, and
\[
 \dom (B^*_kB_k)\cap \dom B^2=\dom (BB_k), \; k=1,2,
 \]
we obtain the equalities
\[
 \dom (BB_1)=\dom (BB_2)=\{0\}.
 \]
The rest of the theorem can be checked similarly to the proof of the corresponding part of Theorem \ref{prodzer}.
\end{proof}
\begin{theorem} \label{brne}
Let $\cB$ be a closed densely defined symmetric operator. Then there are infinitely many pairs
$\left<\cB_1, \cB_2\right>$ of closed densely defined restrictions of $\cB$ such that
\begin{enumerate}
\item $\dom \cB_1\cap\dom \cB_2=\{0\}$;
\item $\dom \cB=\dom \cB_1\dot+\dom \cB_2$;
\item $\dom \cB_1\cap\dom (\cB^*\cB)=\dom \cB_2\cap\dom (\cB^*\cB)=\{0\};$
\item $\dom (\cB^*\cB_1)=\dom (\cB^*\cB_2)=\{0\}$, in particular,  $\dom
\cB^2_1=\dom \cB^2_2=\{0\}$.
\end{enumerate}
{If $\ 0$ is for $\cB$ the point of the regular type: $||\cB f||\ge c||f||$ for some $c>0$ and all $f\in\dom\cB$, then}
\begin{enumerate}
\item $\cB_j\cB^*$ is densely defined,
\item the operator $\cB\cB^*$ is the Friedrichs extension of operators $\cB_1\cB^*$ and $\cB_2\cB^*$,
\item $(\cB_j\cB)^*=\cB\cB^*_j,$ $j=1,2$,
\item $\dom(\cB_1\cB^*)\cap\dom(\cB_2\cB^*)=\ker\cB^*$ and
\[
\dom(\cB_1\cB^*)+\dom (\cB_2\cB^*)=\dom(\cB\cB^*),
\]
\item the Kre\u\i n extension of the operator $\cB_j\cB^*$ is the
operator $\cB_j\cB^*_j$, $ j=1,2.$
\end{enumerate}
\end{theorem}
 \begin{proof}
Let $\cB=UB$ be the polar decomposition of $\cB$. Here $B=(\cB^*\cB)^{1/2}$ and $U$ is a partial isometry
with $\ker U=\ker \cB^*$ and $\ran U=\cran\cB$.  By Theorem \ref{novsch} there is a pair
$\left<B_1, B_2\right>$ of densely defined closed restrictions of
$B$ such that $\dom B=\dom B_1\dot+\dom B_2$, $\dom B_k\cap\dom B^2=\{0\},$ $k=1,2$, and $\dom (BB_1)=\dom (BB_2)=\{0\}$.
Set $\cB_k=U B_k$, $k=1,2$. Then clearly $\cB_1 $ and $\cB_2$ are closed densely defined (and symmetric) restrictions of
$\cB$ with $\dom\cB_1\cap\dom \cB_2=\{0\}$, $\dom \cB=\dom \cB_1\dot+\dom \cB_2$,
and $\dom \cB_k\cap\dom \cB^*\cB=\{0\},$ $k=1,2$.

Since $\cB^*=BU^*$ \cite{Ka}, we have
$$
\cB^*\cB_k=BU^*UB_k=BB_k, \;k=1,2.
$$
Hence, $\dom (\cB^*\cB_1)=\dom (\cB^*\cB_2)=\{0\}$.

Suppose now that $0$ is the point of the regular type for the operator $\cB$. Then it is well-known that there exists
a self-adjoint extension $\wh\cB$ of $\cB$ with $\wh \cB^{-1}\in\bB(\cH)$, and
\[
\begin{array}{l}
\dom\cB^*=\dom\wh \cB\dot+\ker \cB^*,\\
\dom\cB^*_j=\dom\wh \cB\dot+\ker \cB^*_j,
\end{array}
\]
for $j=1,2$. Note that
\[
\begin{array}{l}
\dom(\cB_j\cB^*)=\wh\cB^{-1}\dom\cB_j\dot+\ker \cB^*,\; j=1,2,\\
\dom(\cB\cB^*_j)=\wh\cB^{-1}\dom\cB_j\dot+\ker \cB^*,\; j=1,2,\\
\dom(\cB\cB^*)=\wh\cB^{-1}\dom\cB\dot+\ker\cB^*.
\end{array}
\]
The above equalities and decomposition $\dom \cB=\dom \cB_1\dot+\dom
\cB_2$ leads to $\dom(\cB_1\cB^*)\cap\dom(\cB_2\cB^*)=\ker\cB^*$ and
$\dom(\cB_1\cB^*)+\dom (\cB_2\cB^*)=\dom(\cB\cB^*)$.

Now we show that $\dom(\cB_k\cB^*)$ is dense in $\dom \cB^*$ with respect to the graph inner-product in $\dom\cB^*$.
If for some $h\in\dom\cB^*$ one has
\[
(\cB^*(\wh\cB^{-1}\f_1+\f_0), \cB^*h)+(\wh\cB^{-1}\f_1+\f_0, h)=0
\]
for all $\f_1\in\dom(\cB_1\cB^*)$ and for all $\f_0\in\ker\cB^*$, then
$h=\cB g$ and
\[
(\f_1, \cB\cB^*g+g)=0 \;\mbox{for all}\; \f_1\in\dom\cB_1.
\]
Since $\dom\cB_1$ is dense in $\cH$, we get $g=0$. Thus $\dom(\cB_1\cB^*)$ and similarly $\dom(\cB_2\cB^*)$ are dense in
$\dom \cB^*$ with respect to the graph inner-product in $\dom\cB^*$. This is equivalent to the fact that $\dom(\cB_j\cB^*)$
is dense in $\cH$ and that the Friedrichs extension of $\cB_j\cB^*$ is the operator $\cB\cB^*$ for any of $j=1,2$, i.e.,
$ (\cB_j\cB^*)_{\rm F}=\cB\cB^*.$

Since
\[
\left\{\begin{array}{l} \dom(\cB_1\cB^*)\ni
f=\wh\cB^{-1}\f_1+\psi_0,\; \f_1\in\dom\cB_1,\; \psi_0\in\ker
\cB^*_1,\\
(\cB_1\cB^*)f=\cB_1\f_1 \ ,
\end{array}
\right.
\]
for any $x\in\dom(\cB_1\cB^*)^*$ we get
\[
(\cB_1\f_1, x)=(\wh\cB^{-1}\f_1+\psi_0,(\cB_1\cB^*)^*x)
\]
for all $\f_1\in\dom\cB_1$ and all $\psi_0\in\ker \cB^*_1$. This implies that
\[
\begin{array}{l}
(\cB_1\cB^*)^*x=\cB g,\; g\in\dom\cB_1,\\
x\in\dom\cB^*_1,\; \cB^*_1x=g.
\end{array}
\]
Therefore, if $x\in\dom(\cB_1\cB^*)^*$, then $x\in\dom \cB\cB^*_1$ and $(\cB_1\cB^*)^*x=\cB\cB^*_1x$,
i.e.,
$$(\cB_1\cB^*)^*\subseteq \cB\cB^*_1.$$
On the other hand one has
$$\cB\cB^*_1\subseteq (\cB_1\cB^*)^*.$$
Thus, $(\cB_1\cB^*)^*=\cB\cB^*_1$. Similarly we obtain
$(\cB_2\cB^*)^*=\cB\cB^*_2$. Applying Theorem \ref{Sim} we get that
\[
(\cB_j\cB^*)_{\rm K}=\cB_j\cB^*_j,\; j=1,2.
\]
This completes the proof.
\end{proof}
{Combining Theorem \ref{appl1}, Corollary \ref{nov111}, and Theorems \ref{now11}, \ref{brne} we can summarise them
as the following statement.
\begin{theorem}
\label{nov112} Let $\cB$ be a closed densely defined symmetric operator. Then
\begin{enumerate}
\item for each natural number $n \in \mathbb{N}$ there exists $n$ closed densely
defined restrictions $\{\cB_k\}_{k=1}^{n}$ of the operator $\cB$ such that:
\begin{enumerate}
\item $\dom\cB_k\cap\dom\cB_j=\{0\}$ , $k\ne j,$
\item $\dom \cB_1\dot+\dom\cB_2\dot+\cdots\dot+\dom\cB_n=\dom\cB,$
\item $\dom\cB_k\cap\dom(\cB^*\cB)=\{0\}$ for each $k=1,2,\ldots, n,$
\item $\dom(\cB^*\cB_k)=\{0\}$ for each $k=1,2,\ldots, n,$
\end{enumerate}
\item there exists two infinite sequences $\{\cB_{1,j}\}_{j\in\dN}$ and
$\{\cB_{2,j}\}_{j\in\dN}$ of closed densely defined
 restrictions of the operator $\cB$ such that
\begin{enumerate}
\item
$\cB\supset\cB_{1,1}\supset\cB_{1,2}\supset\cdots\supset\cB_{1,j}\supset\cdots,$
\item
$\cB_{2,1}\subset\cB_{2,2}\subset\cdots\subset\cB_{2,j}\subset\cdots\subset\cB,$
\item $\bigcap_{j\in\dN}\dom \cB_{1,j}=\{0\}$,
\item $\dom\cB_{1,j}\dot+\dom\cB_{2,j}=\dom\cB$,
\item $\dom(\cB^*\cB_{1,j})=\{0\},$ $\dom (\cB^*\cB_{2,j}) =\{0\}$ for
each $j\in\dN$.
\end{enumerate}
\end{enumerate}
\end{theorem}
}

\section{Acknowledgements}\label{Acknow}

\noindent Yu.A. was supported by Aix-Marseilles University during his visit the Laboratoire d'Analyse, Topologie,
Probabilit\'{e}s /Institut de Math\'{e}matiques de Marseille  (UMR 7353).



\begin{thebibliography}{AHS}
\bibitem{And}
W.N.~Anderson, \textit{Shorted operators.} SIAM J. Appl. Math.,
\textbf{20} (1971), 520--525.

\bibitem{AD} W.N. Anderson and R.J. Duffin, \textit{Series and parallel
addition of matrices}, J. Math. Anal. Appl. \textbf{26} (1969),
576-594.

\bibitem{AT}
W.N.~Anderson and G.E.~Trapp, \textit{Shorted operators, II.}  SIAM
J. Appl. Math., \textbf{28} (1975), 60--71.

\bibitem{Ar}
R.~Arens, \textit{Operational calculus of linear relations.} Pacific J.Math., \textbf{11} (1961), 9--23.
J. Appl. Math., \textbf{28} (1975), 60--71.

\bibitem{Ar2}
Yu.M.~Arlinski\u{\i}, \textit{On the theory of operator means}, Ukr.
Mat. Zh., \textbf{42} (1990), No.6, 723--730 (in Russian). English
translation in Ukr. Math. Journ., \textbf{42} (1990), No.6,
639--645.

\bibitem{AHS2}
Yu.M.~Arlinski\u{\i}, S.~Hassi, and H.S.V. de~Snoo,
\textit{Q-functions of Hermitian contractions of Kre\u{\i}n --
Ovcharenko type}, Integr. Equ. Oper. Theory, \textbf{53} (2005),
No.2, 153--189.

\bibitem{ArlKov_2011}
Yu.~Arlinski\u{\i} and Yu.~Kovalev, \textit{Operators in divergence
form and their Friedrichs and Kre\u\i n - von Neumann extensions},
Opuscula Mathematica, \textbf{31} (2011), No.4, 501--517.


\bibitem{ArlKov_2013}
Yu.~Arlinski\u{\i} and Yu.~Kovalev, \textit{Factoriazations of
nonnegative symmetric operators}, Methods of Funct. Anal. and
Topol., \textbf{19} (2013), No.3. 211--226.

\bibitem{ArlZag2010}
Yu.~Arlinski\u{\i} and V.~Zagrebnov, \textit{Numerical range and
quasi-sectorial contractions}, J. Math. Anal. Appl. \textbf{366}
(2010), 33--43.
\bibitem{Bognar}
J.~Bogn\'{a}r, \textit{Indefinite inner product spaces}, Springer-Verlag, Berlin,
1974.

\bibitem{BraNeidh}
J.R.~Brasche and H.~Neidhardt, \textit{Has every symmetric operator
a closed restriction whose square has a trivial domain}? Acta Sci.
Math. (Szeged), \textbf{58} (1993), 425--430.

\bibitem{CNZ}
V.~Cachia, H.~Neidhardt, and V.~Zagrebnov, \textit{Comments on the
Trotter product formula error-bound estimates for nonself-adjoint
semigroups}, Integr. Equ. Oper. Theory, \textbf{42} (2002),
425--448.


\bibitem{Chern}
P.R.~Chernoff, \textit{A semibounded closed symmetric operator whose
square has trivial domain}, Proc. Amer. Math. Soc., \textbf{89}
(1983), 289--290,

\bibitem{CD}
{M.~Crouzeix and B.~Delyon},   
\textit{Some estimates for analytic functions of the strip or a
sectorial ope\-ra\-tors}, {Arch. Math.} \textbf{81}
{(2003)}, {559--566}.


\bibitem{Dix}
J. Dixmier, \textit{Etude sur les varietes et les operateurs de
Julia}, Bull. Soc. Math. France \textbf{77} (1949), 11--101.
\bibitem{Dix2}
J. Dixmier, \textit{L'adjoint du produit dedeux op\'{e}rateurs ferm\'{e}s}, Annales de la facult\'{e} des sciences
de Toulouse 4e S\'{e}rie, \textbf{11} (1947), 101--106.

\bibitem{Doug}
R.G.~Douglas, \textit{On majorization, factorization and range
inclusion of operators in Hilbert space}, Proc. Amer. Math. Soc.
\textbf{17} (1966), 413--416.

\bibitem{EN}
K.-J.~Engel and R.~Nagel, \textit{One parameter semigroups for
linear evolutionsl equations}, Springer-Verlag, Berlin,
Heidelberg, New-York, 1999.

\bibitem{FW}
P.A.~Fillmore and J.P.~Williams, \textit{On operator ranges},
Advances in Math. \textbf{7} (1971), 254--281.

\bibitem{Ka}
T.~Kato, \textit{Perturbation theory for linear operators},
Springer-Verlag, Berlin, 1966.

\bibitem{Ka1}
 T.~Kato, \textit{On the Trotter-Lie product formula}. Proc. Japn. Acad.
\textbf{50} (1974), 694-698.
\bibitem{Ka2}
 T.~Kato, \textit{Trotter's product formula for an arbitrary pair of
self-adjoint contraction semigroups}. Topics in Funct. Anal., Ad.
Math. Suppl. Studies Vot. 3, 185-195 (I.Gohberg and M.Kac eds.).
Acad. Press, New York 1978.

\bibitem{Kosaki}
H.~Kosaki, \textit{On intersections of domains of unbounded positive
operators}, Kyushu J. Math, \textbf{60}, 3--25  (2006).

\bibitem{Kr1}
M.G.~Kre\u{\i}n, \textit{The theory of selfadjoint extensions of
semibounded Hermitian transformations and its applications}, I, Mat.
Sbornik \textbf{20} (1947), No.3, 431--495 (in Russian).

\bibitem{KrO}
M.G. Krein and I.E. Ovcharenko, \textit{On Q-functions and
sc-extensions of a Hermitian contraction with nondense domain},
Sibirsk. Mat. Journ., \textbf{18} (1977), No. 5, 1032-1056 (in
Russian).

\bibitem{KA}
F.~Kubo and T.~Ando, \textit{Means of positive linear operators},
Math. Ann. \textbf{246}(1980), No.3 , 205--224.

\bibitem{Naimark1}
M.A.~Na\u{\i}mark, \textit{On the square of a closed symmetric
operator}, Dokl. Akad. Nauk SSSR, \textbf{26} (1940), 863--867 (in
Russian).

\bibitem{Naimark2}
M.A.~Na\u{\i}mark, \textit{Supplement to the paper "On the square of
a closed symmetric operator"}, Dokl. Akad. Nauk USSR, \textbf{28}
(1940), 206--208 (in Russian).

\bibitem{NeidhZag1}
H.~Neidhardt and V.A.~Zagrebnov  \textit{Does each symmetric
operator have a stability domain?},  Rev. Math. Phys. \textbf{10}
(1998), 829--850.

\bibitem{NeidhZag2}
H.~Neidhardt and V.A.~Zagrebnov  \textit{On semibounded restrictions of self-adjoint operators},
Integr. Equ. Oper. Theory \textbf{31} (1998), 489--512.

\bibitem{Neumann}
J. von Neumann, \textit{Zur Theorie des Unbeschr\"{a}nkten
Matrizen}, J. Reine Angew. Math. \textbf{161} (1929), 208--236.

\bibitem{Neumann1}
J. von Neumann, \textit{Allgemeine Eigenwerttheorie Hermitescher Funktionaloperatoren},
Math. Ann. \textbf{102} (1930), 49--131.

\bibitem{Neumann2}
J. von Neumann, \textit{ Zur Algebra der Funktionaloperatoren und Theorie der normalen Operatoren} ,
Math. Ann. \textbf{102} (1930), 370--427.

\bibitem{Pek}
E.L.~Pekarev, \textit{Shorts of operators and some extremal problems}, Acta Sci. Math.
(Szeged) \textbf{56} (1992), 147--163.

\bibitem{PSh}
E.L. Pekarev and Ju.L. Smul'yan, \textit{Parallel addition and
parallel subtraction of operators}, Izv. AN SSSR \textbf{40} (1976),
366--387 (in Russian). English translation in Math. USSR Izv.
\textbf{10} (1976), 289--337.

\bibitem{schmud} K. Schm\"{u}dgen, \textit{On domains of powers of closed symmetric
operators}, J. Oper. Theory, \textbf{9} (1983), 53--75.

\bibitem{schmud-1} K. Schm\"{u}dgen, \textit{On restrictions of unbounded symmetric operators},
J. Oper. Theory, \textbf{11} (1984), 379--393.

\bibitem{Shmul}
Yu.L.~Shmulyan, \textit{An operator Hellinger integral}, Mat. Sb.
\textbf{49} (1959), 381–-430 (in Russian).

\bibitem{Daele}
A.~Van Daele, \textit{On pairs of closed operators}, Bull. Soc.
Math. Belg., Set. B \textbf{34} (1982), 25--40.

\bibitem{Daele2}
A.~Van Daele, \textit{Dense subalgebras of left Hilbert algebras}, Can. J. Math. \textbf{36} (1982), No.6, 1245--1250.

\bibitem{Zag2008}
V.A.~Zagrebnov  \textit{Quasi-Sectorial Contractions}, Journ. Funct. Anal. \textbf{254} (2008), 2503--2511.

\end{thebibliography}
\end{document}